\documentclass[11pt,english]{amsart}

\usepackage[a4paper]{geometry}
\usepackage{amssymb}
\usepackage{amsmath}
\usepackage{nccmath}
\usepackage{fancyhdr}
\usepackage{color}
\usepackage{enumerate}
\usepackage{tikz}
\newtheorem{theorem}{Theorem}[section] % 1st argument is your name for it
\newtheorem{lemma}[theorem]{Lemma}     % 2nd argument is what is printed

\newtheorem{proposition}[theorem]{Proposition}
\newtheorem{definition}[theorem]{Definition}

\newtheorem{remark}[theorem]{Remark}

\newtheorem*{theorem*}{Theorem}

\usepackage[bookmarksopen, naturalnames]{hyperref}

\makeatother

\begin{document}

%\title[Growth of upper frequently hypercyclic functions ]{ Optimal growth of upper frequently hypercyclic functions for some weighted Taylor shifts}
\title[Growth of hypercyclic functions]{Growth of hypercyclic functions: a continuous path between $\mathcal{U}$-frequent hypercyclicity and hypercyclicity}

\author{A. Mouze, V. Munnier}
\thanks{The first author was partly supported by the grant ANR-17-CE40-0021 of the French
National Research Agency ANR (project Front)}
\address{Augustin Mouze, Univ. Lille, \'Ecole Centrale de Lille, CNRS, UMR 8524 - Laboratoire Paul Painlev\'e  F-59000 Lille, France}
\email{augustin.mouze@univ-lille.fr}
\address{Vincent Munnier, 16 avenue Pasteur, 94100 Saint Maur Des Foss\'es, France}
\email{munniervincent@hotmail.fr}


\begin{abstract} 
We are interested in the optimal growth in terms of $L^p$-averages of hypercyclic and $\mathcal{U}$-frequently hypercyclic functions for some weighted Taylor shift operators acting on the space of analytic functions on the unit disc. We unify the results obtained by considering intermediate notions of upper frequent hypercyclicity between the $\mathcal{U}$-frequent hypercyclicity and the hypercyclicity.     
\end{abstract}
%% maketitle must follow the abstract.
\maketitle                   % Produces the title.

%\footnotetext{The first author was partly supported by the grant ANR-17-CE40-0021 of the French
%National Research Agency ANR (project Front).}
% place in the next line the header (rubrique) chosen for your article,
% if you know it (you can also have 2, format : Header1/Header2

% main text
%\selectlanguage{english}
\section{Introduction} A linear operator on a Fr\'echet space $X$ is said to be \textit{hypercyclic} if there is a vector $x\in X$ such that for every non-empty open 
set $U\subset X$ the set $N(x,U):=\{n\in\mathbb{N}:T^n x\in U\}$ is infinite, where $(T^n)$ is the sequence of iterates of $T$. In this situation, $x$ is called a \textit{hypercyclic} vector. Further there are more precise and stringent notions that allow to quantify how often a hypercyclic vector visits a non-empty open set. A linear operator on a Fr\'echet space $X$ is said to be \textit{frequently hypercyclic} (resp. \textit{$\mathcal{U}$-frequently hypercyclic}) if there is a vector $x\in X$ such that for every non-empty open set $U\subset X$ the set $N(x,U)$ has positive lower (resp. upper) density, where the lower and upper densities of a subset $A\subset \mathbb{N}$ are defined respectively as follows
$$\underline{d}(A)=\liminf_{n\rightarrow +\infty}\frac{\#A\cap\{1,\dots,n\}}{n}\ \hbox{ and }\ 
\overline{d}(A)=\limsup_{n\rightarrow +\infty}\frac{\#A\cap\{1,\dots,n\}}{n}.$$
These notions were introduced by Bayart and Grivaux \cite{BayGriv} and Shkarin \cite{Shka}. The dynamics of linear operators is a very active branch of research: we refer the reader to \cite{BayMath,grossebook} and the references therein for background in linear dynamics. Clearly a frequently hypercyclic vector is $\mathcal{U}$-frequently hypercyclic and a $\mathcal{U}$-frequently hypercyclic vector is hypercyclic. Classical examples of frequently or $\mathcal{U}$-frequently hypercyclic operators are given by suitable weighted shifts. As usual we denote by $\mathbb{D}$ the open unit disc $\{z\in\mathbb{C}:\vert z\vert<1\}$ of the complex plane and by $H(\mathbb{D})$ the set of analytic functions in $\mathbb{D}$. It is well known that $H(\mathbb{D})$ endowed with the topology of uniform convergence on compact subsets is a Fr\'echet space. For $\alpha\in\mathbb{R}$, let $w(\alpha)=(w_n(\alpha))$ be the weighted sequence of nonzero complex numbers given by, for all $n\geq 1$, $$w_n(\alpha)=\left(1+\frac{1}{n}\right)^{\alpha}.$$ 
In the present paper, we consider the associated weighted Taylor shift: 
$$T_{\alpha}:H(\mathbb{D})\rightarrow H(\mathbb{D})\hbox{ given by }T_{\alpha}(\sum_{k\geq 0}a_k z^k)=\sum_{k\geq 0}a_{k+1}w_{k+1}(\alpha) z^k.$$ 
For $\alpha=0$, $T_{0}$ is the classical Taylor shift operator. It is easy to check that for every  real number $\alpha$, $T_{\alpha}$ is a frequently hypercyclic operator. For instance, we refer the reader to \cite{Bernal_CV, ge2, MouMun3}. The problem of determining possible rates of growth of frequently hypercyclic functions for $T_{\alpha}$ in terms of $L^p$ averages was studied in \cite{MouMun3} (see \cite{MouMun2} for the case $\alpha=0$ too). For $0<r<1$ and $f\in H(\mathbb{D})$, we consider the classical integral means  
$$M_p(f,r)=\left(\frac{1}{2\pi}\int_{0}^{2\pi}\vert f(re^{i\theta})\vert^{p}d\theta\right)^{1/p} (1\leq p<\infty)\hbox{ and }
M_{\infty}(f,r)=\sup_{0\leq t\leq 2\pi}\vert f(re^{it})\vert.$$
In the same way, for any holomorphic polynomial $P$ let us define, for $p\geq 1,$ 
$$\Vert P\Vert_p=\left(\frac{1}{2\pi}\int_{0}^{2\pi}\vert P(e^{i\theta})\vert^{p}d\theta\right)^{1/p}\hbox{ and }
\Vert P\Vert_{\infty}=\sup_{0\leq t\leq 2\pi}\vert P(e^{it})\vert.$$
In the following, for all $p>1$ $q$ will stand for the exponent conjugate to $p$, i.e. $\frac{1}{p}+\frac{1}{q}=1$ and we will adopt the convention $q=\infty$ if $p=1$. For $1\leq p\leq\infty$, the authors recently highlighted a \textit{critical exponent}, i.e. a value of the parameter $\alpha$ from which the 
$L^p$-growth of a frequently hypercyclic function for $T_{\alpha}$ no longer has the same behavior. In the case of frequent hypercyclicity for $T_{\alpha}$, the critical exponent is equal to $\alpha=\frac{1}{\max(2,q)}$. Indeed the authors obtained the following statements. First for $p>1$ they proved the following result.

\begin{theorem} {\rm \textbf{(\cite[Theorem 1.2]{MouMun3})}}\label{anc_main_optimal1} Let $\alpha\in\mathbb{R}.$ The following assertions hold
	\begin{enumerate}
		\item For any $1< p< +\infty$ there is a frequently hypercyclic function $f$ in $H(\mathbb{D})$ for $T_{\alpha}$ satisfying the following estimates: there exists $C>0$ such that for every $0<r<1$ 
		$$M_{p}(f,r)\leq \left\{\begin{array}{ll}
		C(1-r)^{\alpha-\frac{1}{\max(2,q)}}& \mbox{ if } \alpha<\frac{1}{\max(2,q)},\\
		C \vert \log(1-r) \vert ^{\frac{1}{p}}& \mbox{ if } \alpha=\frac{1}{\max(2,q)},\\
		C &\mbox{ if } \alpha>\frac{1}{\max(2,q)}.
		\end{array}\right.$$ 
		These estimates are optimal: every frequently hypercyclic function $f$ in $H(\mathbb{D})$ for $T_{\alpha}$ is bounded from below by the corresponding previous estimate depending on $\alpha.$ 
		\item There is a frequently hypercyclic function $f$ in $H(\mathbb{D})$ for $T_{\alpha}$ satisfying the following estimates: 
		there exists $C>0$ such that for every $0<r<1$ 
		$$M_{\infty}(f,r)\leq \left\{\begin{array}{ll}
		C(1-r)^{\alpha-\frac{1}{2}}& \mbox{ if } \alpha<1/2,\\
		C \vert\log(1-r)\vert& \mbox{ if } \alpha=1/2,\\
		C &\mbox{ if } \alpha>1/2.
		\end{array}\right.$$ 
		For $\alpha\ne 1/2,$ these estimates are optimal: every frequently hypercyclic function $f$ in $H(\mathbb{D})$ for $T_{\alpha}$ is bounded from below by the corresponding previous estimate depending on $\alpha.$ 				
	\end{enumerate}
\end{theorem}

For $p=1$, the following result holds. For any positive integer $\ell\geq 1$, $\log_{\ell}$ stands for $\log\circ\dots\circ\log$ where $\log$ appears $\ell$ times.
 
\begin{theorem} {\rm \textbf{(\cite[Proposition 4.1 and Theorem 4.4]{MouMun3})}}\label{anc_main_optimal2}
For any $\ell\geq 1$, there is a frequently hypercyclic function $f$ in $H(\mathbb{D})$ for $T_{\alpha}$ satisfying the following estimates: there exists $C>0$ such that for every $0<r<1$ sufficiently large 
	$$M_{1}(f,r)\leq \left\{\begin{array}{ll}
	C(1-r)^{\alpha}\log_\ell(-\log(1-r))& \mbox{ if } \alpha<0,\\
	C \vert \log(1-r) \vert \log_\ell(-\log(1-r))& \mbox{ if } \alpha=0,\\
	C &\mbox{ if } \alpha>0.
	\end{array}\right.$$ 
	Moreover every frequently hypercyclic function $f$ in $H(\mathbb{D})$ for $T_{\alpha}$ satisfies the following estimates:
	$$\liminf_{r\rightarrow 1^{-}}\left[M_{1}(f,r)(1-r)^{-\alpha}\right]>0 \mbox{ if } \alpha<0,\quad  
	\displaystyle\liminf_{r\rightarrow 1^{-}}\left[\frac{M_{1}(f,r)}{-\log(1-r)}\right]>0 \mbox{ if } \alpha=0, $$
		$$\displaystyle\liminf_{r\rightarrow 1^{-}}\left[M_{1}(f,r)\right]>0\mbox{ if } \alpha>0.$$ 
\end{theorem}

It should be noted that the study of the growth of hypercyclic or frequently hypercyclic functions started with those related to the differentiation operator on $H(\mathbb{C})$ (see for instance \cite{BBG,Drasin, ge1, ge11}) but was also recently extended to the partial differentiation operator \cite{GST} or the Dunkl operator \cite{Ber2016}. Here, as a first step, we obtain sharp results on the permissible rates of $L^p$-growth of hypercyclic and $\mathcal{U}$-frequently hypercyclic functions for $T_{\alpha}$. On one hand, for hypercyclicity, for any $1\leq p\leq \infty$ we find that the rate of growth $(1-r)^{\min(\alpha,0)}$ turns out to be critical and hence, for any $1\leq p\leq \infty$, $\alpha=0$ is the critical exponent. Observe that in this case the critical exponent does not depend on $p$. We refer to Theorem \ref{hcalpha}. In particular this result states that for $1\leq p\leq\infty$ there is no hypercyclic function $f$ for $T_{\alpha}$ satisfying $\limsup_{r\rightarrow 1^{-}}((1-r)^{-\alpha}M_p(f,r))<+\infty$ if $\alpha\leq 0$ while if $\alpha>0$ there exist hypercyclic functions $f$ for $T_{\alpha}$ such that the average $M_p(f,r)$ is bounded. In passing Theorem \ref{hcalpha} gives a negative answer to a question of \cite{MouMun3} which asked if for $\alpha <0$ there is a frequently hypercyclic function $g_{\alpha}$ for $T_{\alpha}$ such that $\limsup_{r\rightarrow 1^-}( (1-r)^{-\alpha}M_1(g_{\alpha},r))<+\infty$. On the other hand, for $\mathcal{U}$-frequent hypercyclicity, we find the same critical exponent $\alpha=\frac{1}{\max(2,q)}$ as for frequent hypercyclicity. Therefore contrary to the previous case this exponent depends on $p$. Moreover we show that the $\mathcal{U}$-frequently hypercyclic functions and the frequently hypercyclic functions for $T_{\alpha}$ share the same admissible (and optimal) $L^p$-growth when $\alpha$ is different from the critical exponent, i.e. $\alpha\ne \frac{1}{\max(2,q)}$. Concerning the case $\alpha= \frac{1}{\max(2,q)}$ with $1\leq p\leq\infty$, we prove that every $\mathcal{U}$-frequently hypercyclic vector $f$ for $T_{\alpha}$ satisfies $\limsup_{r\rightarrow 1^{-}}M_p(f,r)=+\infty$, without a priori additional information on the growth of the function. We refer to Theorems \ref{ufhcalpha_no_critic} and \ref{ufhcalpha_no_criticp1}. Nevertheless several questions remain and need to be addressed. The first question that comes to mind is the following: what is the optimal boundary growth of $\mathcal{U}$-frequently hypercyclic functions for $T_{\alpha}$ when $\alpha$ is the critical exponent? Further if we go back to what was just said, we see that for $p=1$ the critical exponent is always equal to $0$ for the hypercyclic case, the $\mathcal{U}$-frequently case and the frequently hypercyclic case. But surprisingly for $p>1$ this critical exponent is equal to $\frac{1}{\max(2,q)}$ for the $\mathcal{U}$-frequently or frequently hypercyclic cases and is equal to zero for the hypercyclic case. Thus, as a second question, we can wonder about what happens between $\mathcal{U}$-frequent hypercyclicity and hypercyclicity. Why does the critical exponent go from $\frac{1}{\max(2,q)}$ to zero?  In order to understand this phenomenon, we introduce intermediate notions of linear dynamics between the $\mathcal{U}$-frequent hypercyclicity and the hypercyclicity: the $\mathcal{U}_{\beta^{\gamma}}$-frequent hypercyclicity related to notions of upper  weighted densities $\overline{d}_{\beta^{\gamma}}$, with $0\leq \gamma\leq 1$ a continuous parameter, where we replace in the definition of $\mathcal{U}$-frequent hypercyclicity the natural upper density $\overline{d}$ by $\overline{d}_{\beta^{\gamma}}$. Moreover for $\gamma=0$ the  $\mathcal{U}_{\beta^{0}}$-frequent hypercyclicity will be the frequent hypercyclicity and for $\gamma=1$ the  $\mathcal{U}_{\beta^{1}}$-frequent hypercyclicity will be the hypercyclicity. Further for any $0\leq \gamma\leq\gamma'\leq 1$ and for any subset $E\subset \mathbb{N}$, the following chain of inequalities $\overline{d}(E)\leq \overline{d}_{\beta^{\gamma}}(E)\leq\overline{d}_{\beta^{\gamma'}}(E)\leq\overline{d}_{\beta^{1}}(E)$ will show that the $\mathcal{U}_{\beta^{\gamma}}$-frequent hypercyclicity for $\gamma\in(0,1)$ furnish refined notions of linear dynamics between the $\mathcal{U}$-frequent hypercyclicity and the hypercyclicity. We refer the reader to the beginning of Section \ref{section_ubgamma} for the main definitions and properties. Similar notions of weaker densities have been recently studied in the context of linear dynamics (see for instance \cite{BMenetPP, BGr, ErEsMenet, ErMo1, ErMo2, Menet} and the references therein). In the present paper, we investigate the growth in terms of $L^p$-averages  of $\mathcal{U}_{\beta^{\gamma}}$-frequently hypercyclic functions for $T_{\alpha}$. In particular, for  $0<\gamma<1$, and for $p>1$ we find that the critical exponent is given by $\alpha =\frac{1-\gamma}{\max(2,q)}$. Hence let us observe that this critical exponent:
\begin{enumerate}[\textbullet]
\item tends to $\frac{1}{\max(2,q)}$ as $\gamma$ tends to zero, i.e. tends to the critical exponent for the $\mathcal{U}$-frequent hypercyclicity case;
\item tends to $0$ as $\gamma$ tends to $1$, i.e. tends to the critical exponent for the hypercyclicity case.
\end{enumerate}
These estimates thus allow to highlight \textit{a continuous path} between the rate of growth of hypercyclic and $\mathcal{U}$-frequently hypercyclic functions: the growth (in terms of $L^p$-averages) of a hypercyclic function for $T_{\alpha}$ continuously depends on the frequency of visits (measured by the densities $\overline{d}_{\beta^\gamma}$, $0\leq\gamma\leq 1$) of non-empty open subsets by its orbit under the action of $T_{\alpha}$. We also show that the estimates on the growth of $\mathcal{U}_{\beta^{\gamma}}$-frequently hypercyclic functions that we obtained are optimal. To do this, we apply a method based on the use of Rudin-Shapiro polynomials and inspired by a construction of frequently hypercyclic functions with optimal growth for differentiation operator on $H(\mathbb{C})$ due to Drasin and Saksman \cite{Drasin} and that has also been adapted for the proofs of Theorems \ref{anc_main_optimal1} and \ref{anc_main_optimal2} in \cite{MouMun2, MouMun3}. For all these results, we refer the reader to Theorems \ref{ubetafhcalpha_no_critic}, \ref{thmubetaopti} and \ref{udfhcalpha_no_criticp1}. Finally let us return to the first question mentioned above. In the last section, we answer it by showing that the optimal growth of $\mathcal{U}$-frequently and $\mathcal{U}_{\beta^{\gamma}}$-frequently hypercyclic functions for $T_{\alpha}$ coincides whenever $\alpha$ is the critical exponent: actually the $L^p$-growth can be arbitrarily slow as in the hypercyclic case. We refer to Theorem \ref{mainexpcritical}.\\

The paper is organized as follows. In Sections \ref{sectionhc} and \ref{sectionufhc} we establish the boundary behavior of hypercyclic functions and $\mathcal{U}$-frequently hypercyclic functions for $T_{\alpha}$ respectively. In Section \ref{section_ubgamma} we deal with the $\mathcal{U}_{\beta^{\gamma}}$-frequently hypercyclic functions for $T_{\alpha}$. In Section \ref{sectioncritic} we turn our attention to the specific case of critical exponent.\\

Throughout the paper, whenever $A$ and $B$ depend on some parameters, we will use the notation $A\lesssim B$ (resp. $A\gtrsim B$) to mean $A\leq CB$ (resp. $A\geq C B$) for some constant $C>0$ that does not depend on the involved parameters.

\section{Growth of hypercyclic functions}\label{sectionhc} In this section, we are going to establish the rate of growth of hypercyclic functions with respect to the weighted Taylor shift operator $T_{\alpha}$. To do this, inspired by the proofs of \cite[Theorem (A)]{ge1} and \cite[Theorem 3]{Ber2016}, where the authors are interested in the rate of growth of hypercyclic functions with respect to the Mac-Lane operator or the Dunkl operator respectively, we need an important tool in linear dynamics: the Universality Criterion.  Indeed a natural extension of the notion of hypercyclicity is the concept of universality. A sequence of continuous linear mappings $L_n:X\rightarrow Y$ between topological vector spaces $X,Y$ is said to be \textit{universal} whenever there exists a vector $x\in X$ such that the set $\{L_nx\ ;\ n\in\mathbb{N}\}$ is dense in $Y$. Such a vector $x$ is called a \textit{universal} vector for $(L_n)$. Observe that an operator $T:X\rightarrow X$ is hypercyclic if and only if the sequence $(T^n)$ is universal. The following result which is known as the Universality Criterion furnishes a sufficient condition for universality \cite{gethesis}. It is a refined version of Hypercyclicity Criterion \cite{gs,CKitai}.  

\begin{theorem} {\rm \textbf{(Universality Criterion)}} Assume that $X$ and $Y$ are topological vector spaces, such that $X$ is a Baire space and $Y$ is separable and metrizable. Let $L_j:X\rightarrow Y$ be a sequence of continuous linear mappings. Suppose that there are dense subsets $X_0$ of $X$ and $Y_0$ of $Y$ and mappings $S_j:Y_0\rightarrow X$ such that 
	\begin{enumerate}[(i)]
		\item for every $x\in X_0$, $L_jx\rightarrow 0$, 
		\item for every $y\in Y_0$, $S_jy\rightarrow 0$, 
		\item for every $y\in Y_0$, $(L_jS_j)y\rightarrow y$.
	\end{enumerate}
Then $(L_j)$ is universal and the set of universal vectors for $(L_j)$ is residual in $X$.
	
	\end{theorem}

Now we are ready to obtain the critical rate of growth for hypercyclic functions with respect to the weighted Taylor shift operator $T_{\alpha}$. The following statement holds. 

\begin{theorem} \label{hcalpha}
	Let $1\leq p\leq\infty$. \begin{enumerate} \item Let $\alpha\leq 0.$
		\begin{enumerate}
		\item\label{assa} For any function $\varphi:[0,1)\rightarrow\mathbb{R}_+$ with $\varphi(r)\rightarrow \infty$ as $r\rightarrow 1^{-}$ there is a hypercyclic function $f$ for $T_{\alpha}$ with 
		$$M_p(f,r)\lesssim\varphi(r) (1-r)^{\alpha}\quad\hbox{for }0<r<1\hbox{ sufficiently close to }1.$$
		\item\label{assb} There is no hypercyclic function $f$ for $T_{\alpha}$ that satisfies, for $0<r<1$ 
		$$M_p(f,r)\leq C (1-r)^{\alpha},$$
		where $C>0$.
	\end{enumerate}
\item Let $\alpha>0$.
\begin{enumerate}
	\item \label{assa2}There is a hypercyclic function $f$ for $T_{\alpha}$ with
	$$M_p(f,r)\leq C$$
	for some $C>0$.
	\item \label{case2bhcalpha}For any function $\varphi:[0,1)\rightarrow\mathbb{R}_+$ with $\varphi(r)\rightarrow 0$ as $r\rightarrow 1^{-}$, there is no hypercyclic function $f$ for $T_{\alpha}$ that satisfies, for $0<r<1$ 
	$$M_p(f,r)\leq \varphi(r).$$
\end{enumerate}
\end{enumerate}

\end{theorem}

\begin{proof} Since we have for $1\leq p\leq l$
	$$M_p(f,r)\leq M_{l}(f,r),\quad \hbox{ for } 0<r<1,$$ 
	it suffices to prove assertions (\ref{assa}) and (\ref{assa2}) for $M_{\infty}(f,r)$ and assertions (\ref{assb}) and (\ref{case2bhcalpha}) for $M_1(f,r)$. \\
	
	\begin{enumerate}
		\item We begin by the case $\alpha\leq 0$. \\ 
	First we can assume without loss of generality that the function $\varphi$ is increasing and continuous with $\varphi(0)>0$. Let us consider the space $X$ of all functions $f$ in $H(\mathbb{D})$ with $f(z)=\sum\limits_{k\geq 0}a_k z^k$ satisfying for every $n\geq 0$ $\rho_n(f)<+\infty$ and $\rho_n(f)\rightarrow 0$ as $n\rightarrow +\infty$, where 
	$$\rho_n(f)=\sup_{\vert z\vert < 1}\left\{\left\vert\sum_{k=n}^{+\infty}a_k z^k\right\vert (1-\vert z\vert)^{-\alpha}[\varphi(\vert z\vert)]^{-1}\right\}.$$
	It is easy to check that $X$ endowed with the norm $\Vert .\Vert=\sup_{n} \rho_n(.)$ is a Banach space. Therefore $(X,\Vert .\Vert)$ is a Baire space. For all integer $j$, let $L_j:X\rightarrow H(\mathbb{D})$ be the operator given by $L_jf=T^jf$. Clearly $(L_j)$ is a sequence of continuous linear operators. We choose $X_0=Y_0=\mathcal{P}$ the set of polynomials. The set $\mathcal{P}$ is dense in $H(\mathbb{D})$. Moreover, setting the polynomial $s_N(f)=\sum_{k=0}^N a_k z^k$ we get 
	$$\rho_n(f-s_N(f))=\left\{\begin{array}{ll}\rho_n(f)& \hbox{ for }n\geq N+1\\
	\rho_{N+1}(f)& \hbox{ for }n\leq N\end{array}\right.$$ 
	which implies $\Vert f-s_N(f)\Vert=\sup_{n\geq N+1}\rho_n(f)\rightarrow 0$ as $N$ tends to infinity. Hence $\mathcal{P}$ is dense in $X$. Then we define the operators $S_j$ as follows
	$$S_j:\mathcal{P}\rightarrow X,\quad S_j(\sum_{k=0}^n a_k z^k)=\sum_{k=0}^n a_k\frac{(k+1)^{\alpha}}{(k+j+1)^{\alpha}} z^{k+j}.$$
	Clearly we have, for all $P\in \mathcal{P}$, 
	$$L_j(P)\rightarrow 0,\hbox{ as }j\rightarrow +\infty,\quad \hbox{ and } L_jS_j(P)=P.$$
	Now we prove that, for all $P\in \mathcal{P},$ $S_j(P) \rightarrow 0$, as $j\rightarrow +\infty$. Since $S_j(z^k)=(k+1)^{\alpha}S_{j+k}(1)$, it suffices to show that $S_j(1)\rightarrow 0$, as $j\rightarrow +\infty$. To do this, observe that
	$$\Vert S_j(1)\Vert=\sup_{0<r<1}\frac{r^j(1-r)^{-\alpha}}{(j+1)^{\alpha}\varphi(r)}.$$ 
	Let us define $h_j:[0,1)\rightarrow \mathbb{R}_+$ given by $h_j(r)=\frac{r^j(1-r)^{-\alpha}}{\varphi(r)}$. We have  $h_j(0)=0=\lim\limits_{r\rightarrow 1^{-}} h_j(r)$. Let $0<r_j<1$ with $h_j(r_j)=\sup_{0<r<1}\frac{r^j(1-r)^{-\alpha}}{\varphi(r)}$. If $r_{j+1}< r_j,$ we get 
	$$h_{j+1}(r_{j+1})=r_{j+1}h_j(r_{j+1})< r_j h_j(r_{j+1})\leq r_j h_j(r_j)=h_{j+1}(r_j)$$
	which gives a contradiction. Hence the sequence $(r_j)$ is increasing. If $r_j\rightarrow \gamma$ with $\gamma<1$, then 
	$$\Vert S_j(1)\Vert=(j+1)^{-\alpha}h_j(r_j)\leq (j+1)^{-\alpha}\frac{\gamma^j}{\varphi(0)}\rightarrow 0,\hbox{ as }j\rightarrow +\infty.$$
	Otherwise $r_j\rightarrow 1$ and 
	$$\Vert S_j(1)\Vert\leq \frac{(j+1)^{-\alpha}}{\varphi(r_j)}\left(\frac{j}{j-\alpha}\right)^j
	\left(1-\frac{j}{j-\alpha}\right)^{-\alpha}\rightarrow 0,\hbox{ as }j\rightarrow +\infty.$$
	Thus we have $\Vert S_j(1)\Vert\rightarrow 0$ as $j$ tends to infinity. We apply the universality criterion to obtain universal elements for the sequence $(L_j)$ that are hypercyclic functions for $T_{\alpha}$ satisfying the growth condition required. 
	\\
	
	For assertion (\ref{assb}), assume that $f=\sum\limits_{k\geq 0}a_kz^k$ is a function in $H(\mathbb{D})$ with, for all $0<r <1$, $M_1(f,r)\leq C (1-r)^{\alpha}$, for some $C>0$. By Cauchy estimates we get 
	$$\vert a_n\vert\leq \frac{M_1(f,r)}{r^n}.$$
	We obtain, for all $n\geq 0$ and all $0<r<1$,
	$$\vert a_n w_1(\alpha)\dots w_n(\alpha)\vert\leq C \frac{\vert w_1(\alpha)\dots w_n(\alpha)\vert}{r^n}(1-r)^{\alpha}.$$ 
	Hence we get for all $n\geq 0$,
	$$\vert a_n (n+1)^{\alpha}\vert\leq C \frac{(n+1)^{\alpha}}{e^{-n/(n+1)}}(1-e^{-1/(n+1)})^{\alpha}$$ 
	which is bounded. Hence $f$ cannot be hypercyclic for $T_{\alpha}$.\\
	
\item Now let us consider the case $\alpha>0$.\\ 
First we can assume without loss of generality that the function $\varphi$ is increasing and continuous with $\varphi(0)>0$. Let us consider the Banach space  $(H^{\infty}(\mathbb{D}),\Vert .\Vert)$ 
$$H^{\infty}(\mathbb{D})=\left\{f\in H(\mathbb{D})\ ;\  \Vert f\Vert:=\sup_{0<r<1}M_{\infty}(f,r)<\infty      \right\}$$
which is continuously embedded in $H(\mathbb{D})$. We set $X$ the closure of the polynomials in $H^{\infty}(\mathbb{D})$. Let us define the sequences $(L_j)$ and $(S_j)$ of linear operators as in the previous case. 
Clearly we have, for all $P\in \mathcal{P}$, where $\mathcal{P}$ is the set of polynomials, 
$$L_j(P)\rightarrow 0,\hbox{ as }j\rightarrow +\infty,\quad \hbox{ and } L_jS_j(P)=P.$$
Now we prove that, for all $P\in \mathcal{P},$ $S_j(P) \rightarrow 0$, as $j\rightarrow +\infty$. To do this, it suffices to show that $S_j(1)\rightarrow 0$, as $j\rightarrow +\infty$. Since  
$$\Vert S_j(1)\Vert\leq\frac{1}{(j+1)^{\alpha}}$$ 
and $\alpha>0$ we have $\Vert S_j(1)\Vert\rightarrow 0$ as $j$ tends to infinity. We apply the universality criterion to obtain universal elements for the sequence $(L_j)$. These universal vectors are clearly hypercyclic functions for $T_{\alpha}$ satisfying the growth condition required. 
\\

For assertion (\ref{case2bhcalpha}), assume that $f=\sum\limits_{k\geq 0}a_kz^k$ is a function in $H(\mathbb{D})$ with, for all $0<r <1$, $M_1(f,r)\leq \varphi(r)$, where  $\varphi:[0,1)\rightarrow\mathbb{R}_+$ is a function such that $\varphi(r)\rightarrow 0$ as $r\rightarrow 1^{-}$. We obviously may assume that $\varphi$ is continuous and decreasing. By Cauchy estimates we get 
$$\vert a_n\vert\leq \frac{\varphi(r)}{r^n}.$$
We obtain, for all $n\geq 0$ and all $0<r<1$,
$$\vert a_n (n+1)^{\alpha}\vert\leq \frac{(n+1)^{\alpha}\varphi(r)}{r^n}.$$  
Let us choose a sequence $(r_n)$ such that $r_n\geq \max(1-1/n,\varphi^{-1}\left((n+1)^{-\alpha}\right))$. Hence we get for all $n\in\mathbb{N}$,
$$\vert a_n (n+1)^{\alpha}\vert\leq \frac{(n+1)^{\alpha}\varphi(r_n)}{r_n^n}\leq e^{-n\log(1-1/n)}$$ 
which is bounded. Hence $f$ cannot be hypercyclic for $T_{\alpha}$.
	\end{enumerate}
\end{proof}

\begin{remark}{\rm For $\alpha< 0$, Theorem \ref{hcalpha} ensures that every hypercyclic function $f$ for $T_{\alpha}$ satisfies 
		$$\displaystyle\limsup_{r\rightarrow 1^{-}}\left[(1-r)^{-\alpha}M_1(f,r)\right]=+\infty.$$ 
		Since a frequent hypercyclic function is necessarily hypercyclic, this observation gives a negative answer to the first part of the question from Remark 4.5 of \cite{MouMun3} which asked if for $\alpha <0$ there is a frequently hypercyclic function $g_{\alpha}$ for $T_{\alpha}$ such that $\limsup\limits_{r\rightarrow 1^-}( (1-r)^{-\alpha}M_1(g_{\alpha},r))<+\infty$. 
}
	
\end{remark}

\vskip10mm

\section{Growth of $\mathcal{U}$-frequently hypercyclic functions} \label{sectionufhc} 
In this section we are interested in the growth of $\mathcal{U}$-frequently hypercyclic functions for $T_{\alpha}$. First of all, in the sequel we will need the following easy lemmas. 

\begin{lemma}\label{lemme_estim_preli} Let $N\in\mathbb{N}$. Let $A_N$ be a subset of $\{1,\dots,N\}$. For all $\gamma\in\mathbb{R}\setminus\{1\}$ the following estimate holds
	$$\sum_{k\in A_N} (k+1)^{\gamma}\geq\left\{\begin{array}{ll}
	\displaystyle\frac{(\#A_N+1)^{\gamma +1}-1}{\gamma+1}&\hbox{ if }\gamma \geq 0,\\
	\displaystyle\frac{(N+2)^{\gamma+1}}{\gamma+1}\left(1-\left(1-\frac{\#A_N}{N+2}\right)^{\gamma+1}\right)&\hbox{ if }\gamma < 0.\end{array}\right.
	$$
\end{lemma}

\begin{proof} We begin by the case $\gamma \geq 0$. We write
	$$\sum_{k\in A_N} (k+1)^{\gamma}\geq \sum_{k=1}^{\#A_N} (k+1)^{\gamma}\geq \int_0^{\#A_N}(t+1)^{\gamma}dt,$$
	which gives the announced result.\\
	For $\gamma<0$ with $\gamma\ne -1$, we obtain in an analogue way
	$$\sum_{k\in A_N} (k+1)^{\gamma}\geq \sum_{k=N-\#A_N+1}^{N} (k+1)^{\gamma}\geq \int_{N-\#A_N+1}^{N+1}(t+1)^{\gamma}dt,$$
	which allows to finish the proof.	
\end{proof}

\begin{lemma}\label{Abel_Transform} Let $(u_k)$ and $(v_k)$ be two sequences of non-negative real numbers. Assume that $(v_k)$ is decreasing. For any increasing sub-sequence $(N_j)\subset\mathbb{N}$, the following inequality holds, for all $l\geq 1$:
	$$\sum_{k=1+N_0}^{N_l}u_kv_k\geq S_{N_l}v_{N_l}-S_{N_0}v_{N_0}+\sum_{j=1}^l S_{N_{j-1}}(v_{N_{j-1}}-v_{N_j}),$$
	with $S_N=\sum\limits_{k=1}^N u_k$.
\end{lemma}

\begin{proof} Observe that $(S_N)$ is increasing and $(v_k)$ is decreasing. Thus by using a summation by parts we derive
	$$\begin{array}{rcl}\displaystyle\sum_{k=1+N_0}^{N_l}u_kv_k=\sum_{j=1}^l\sum_{1+N_{j-1}}^{N_j}u_kv_k&=&\displaystyle\sum_{j=1}^l\left(S_{N_j}v_{N_j}-S_{N_{j-1}}v_{N_{j-1}}+\sum_{k=N_{j-1}}^{N_j-1}S_k(v_k-v_{k+1})\right)\\
	&=&\displaystyle S_{N_l}v_{N_l}-S_{N_0}v_{N_0}+\sum_{j=1}^l\sum_{k=N_{j-1}}^{N_j-1}S_k(v_k-v_{k+1})\\&\geq & \displaystyle S_{N_l}v_{N_l}-S_{N_0}v_{N_0}+\sum_{j=1}^l S_{N_{j-1}}(v_{N_{j-1}}-v_{N_j}).\end{array}$$
\end{proof}

We are ready to establish the boundary behavior of $\mathcal{U}$-frequently hypercyclic functions for $T_{\alpha}$. Actually we are going to prove that these functions share the same optimal growth as frequently hypercyclic functions except in the case where $\alpha$ is the critical exponent, for which we will show in Section \ref{sectioncritic} that the growth can be arbitrarily slow. We begin by the case $p>1$.

\begin{theorem}\label{ufhcalpha_no_critic}
	Let $f$ be a $\mathcal{U}$-frequently hypercyclic function for the operator $T_{\alpha}$ and $1< p\leq\infty$. Then the following estimates hold
	\begin{align*}&\limsup_{r\rightarrow 1^{-}}\left((1-r)^{\frac{1}{\max(2,q)}-\alpha}M_p(f,r)\right)>0,\quad\hbox{ if }\alpha<\frac{1}{\max(2,q)},\\
	&\limsup_{r\rightarrow 1^{-}}\left(M_p(f,r)\right)=+\infty,\quad\hbox{ if }\alpha=\frac{1}{\max(2,q)},\\
	&\limsup_{r\rightarrow 1^{-}}M_p(f,r)>0,\quad\hbox{ if }\alpha>\frac{1}{\max(2,q)}.
	\end{align*}
	For $\alpha\ne\frac{1}{\max(2,q)}$, these results are optimal in the following sense: for all $p>1$  there exists a $\mathcal{U}$-frequently hypercyclic function for $T_{\alpha}$ such that, for every $0<r<1$, 
	$$M_p(f,r)\lesssim\left\{\begin{array}{ll} (1-r)^{\alpha-\frac{1}{\max(2,q)}}& \hbox{ if } \alpha<\frac{1}{\max(2,q)}\\1&
	\hbox{ if } \alpha>\frac{1}{\max(2,q)}.\end{array}\right.$$
\end{theorem}

\begin{proof}
We write $f=\sum\limits_{k\geq 0}\frac{a_k}{(k+1)^{\alpha}}z^k.$ Since $f$ is $\mathcal{U}$-frequently hypercyclic there exists an increasing sub-sequence $(n_k)\subset\mathbb{N}$ with positive upper density such that for all $k\geq 1$
$$\vert T_{\alpha}^{n_k}f(0)-3/2\vert=\vert a_{n_k}-3/2\vert <1/2.$$
We get, for all $k\geq 1$, $\vert a_{n_k}\vert \geq 1$. Set $I=((n_k))$ and for all $N\geq 1$, $I_N=I\cap\{1,\dots,N\}$. The hypothesis $\overline{d}(I)>0$ ensures that there exist $0<C<1$ and an increasing sequence $(N_l)$ of positive integers such that 
\begin{equation}\label{equlimsup}
\#I_{N_l}\geq C N_l.
\end{equation}
Up to take a sub-sequence, we can also assume that 
\begin{equation}\label{equlimsup00}
C (N_{l+1}+1)\geq N_l +1 .
\end{equation}
Let us consider, for all $l\geq 1$, $1-\frac{1}{N_l-1}\leq r_l<1-\frac{1}{N_l}$. Thus we derive
\begin{equation}\label{encNl}
N_l-1\leq \frac{1}{1-r_l}<N_l.
\end{equation}
	\begin{enumerate}
		\item \textbf{Case $2\leq p\leq\infty$:}\\
		Jensen'inequality and Parseval's Theorem give
		$$[M_p(f,r_l)]^2\geq [M_2(f,r_l)]^2=\sum_{k\geq 0}\frac{\vert a_k\vert^2}{(k+1)^{2\alpha}}r_l^{2k}\geq \sum_{k=1}^{N_l}\frac{\vert a_k\vert^2}{(k+1)^{2\alpha}}r_l^{2k}.$$
	Thus we deduce
		\begin{equation}\label{minor_princ}
		[M_p(f,r_l)]^2\geq \displaystyle\left(1-\frac{1}{N_l-1}\right)^{2N_l}\sum_{k=1}^{N_l}\frac{\vert a_k\vert^2}{(k+1)^{2\alpha}}
		\gtrsim  \sum_{k=1}^{N_l}\frac{\vert a_k\vert^2}{(k+1)^{2\alpha}}
		\end{equation}
	and using the inequality $\vert a_k\vert\geq 1$ for $k\in I_{N_l}$
	\begin{equation}\label{minor_trivial}[M_p(f,r_l)]^2\gtrsim  \sum_{k\in I_{N_l}}\frac{1}{(k+1)^{2\alpha}}.
	\end{equation}
	
		\begin{enumerate}
			\item \textit{Case $\alpha \leq 0$}: Combining Lemma \ref{lemme_estim_preli} with (\ref{equlimsup}), (\ref{encNl}) and (\ref{minor_trivial}) we get
				$$[M_p(f,r_l)]^2\gtrsim  N_l^{-2\alpha +1}\geq (1-r_l)^{2\alpha-1}.$$
				
				\item \textit{Case $0<\alpha <1/2$}: using (\ref{minor_trivial}) and Lemma \ref{lemme_estim_preli} again, we get 
					$$[M_p(f,r_l)]^2 \geq \frac{(2+N_l)^{-2\alpha+1}}{-2\alpha+1}\left(1-\left( 1-\frac{\#I_{N_l}}{2+N_l}\right)^{-2\alpha+1}\right).$$ 
The inequality (\ref{equlimsup}) ensures $\left( 1-\frac{\#I_{N_l}}{2+N_l}\right)^{-2\alpha+1}\leq\left( 1-\frac{CN_l}{2+N_l}\right)^{-2\alpha+1}$. We deduce by using (\ref{equlimsup}) and (\ref{encNl}) again	
$$[M_p(f,r_l)]^2\gtrsim  N_l^{-2\alpha +1}\geq (1-r_l)^{2\alpha-1}.$$
Hence we conclude 
$$\limsup_{r\rightarrow 1^{-}}\left((1-r)^{\frac{1}{2}-\alpha}M_p(f,r)\right)>0.$$
				
	\item \textit{Case $\alpha=\frac{1}{2}$}: using (\ref{minor_princ}), we know
	$$ [M_p(f,r_l)]^2\gtrsim \sum_{j=1}^{l}\sum_{k=1+N_{j-1}}^{N_j}\frac{\vert a_k\vert^2}{k+1}.$$
Hence applying Lemma \ref{Abel_Transform} with $u_k=\vert a_k\vert^2$ and $v_k=1/(k+1)$, we get, for all $l\geq 1$, 
$$[M_p(f,r_l)]^2\gtrsim	\frac{S_{N_l}}{N_l+1}-\frac{S_{N_0}}{N_0+1}+\sum_{j=1}^l S_{N_{j-1}}\left(\frac{1}{N_{j-1}+1}-\frac{1}{N_j+1}\right).$$
By construction $S_{N_{k-1}}\gtrsim N_{k-1}+1$. Thus taking into account (\ref{equlimsup00}) we derive, for all $l\geq 1$, the inequality 
$$ 	[M_p(f,r_l)]^2\gtrsim \sum_{j=1}^l\left(1-\frac{N_{j-1}+1}{N_j+1}\right)\gtrsim l,$$
that allows to obtain $\displaystyle\limsup_{r\rightarrow 1^{-}}M_p(f,r)=+\infty$. 	
\end{enumerate}			
		\item \label{casnonrefait1} \textbf{Case $1< p < 2$:}\\
		From Hausdorff-Young inequality (see  \cite{PDuren}) we get
		$$[M_p(f,r_l)]^q\geq \sum_{k\geq 0}\frac{\vert a_k\vert^q}{(k+1)^{q\alpha}}r_l^{qk}
		\geq \sum_{k=1}^{N_l}\frac{\vert a_k\vert^q}{(k+1)^{q\alpha}}r_l^{qk}.$$
		Thus we deduce
		$$[M_p(f,r_l)]^q\geq \displaystyle\left(1-\frac{1}{N_l-1}\right)^{qN_l}\sum_{k=1}^{N_l}\frac{\vert a_k\vert^q}{(k+1)^{q\alpha}}
		\gtrsim  \sum_{k=1}^{N_l}\frac{\vert a_k\vert^q}{(k+1)^{q\alpha}}.$$ 
		Using the same strategy as in the case $2\leq p\leq\infty$, we obtain, 
		\begin{align*}&\limsup_{r\rightarrow 1^{-}}\left((1-r)^{\frac{1}{q}-\alpha}M_p(f,r)\right)>0,\quad \hbox{ for } \alpha<1/q,\\
		&\limsup_{r\rightarrow 1^{-}}\left(M_p(f,r)\right)=+\infty,\quad \hbox{ for } \alpha=1/q.
		\end{align*}
	\end{enumerate}

	Moreover, since a $\mathcal{U}$-frequently hypercyclic function is necessarily hypercyclic, the assertion for the case $\alpha>\frac{1}{\max(2,q)}$ of the statement is given by the assertion (\ref{case2bhcalpha}) of Theorem \ref{hcalpha}.
\\
Finally, since a frequently hypercyclic function is necessarily $\mathcal{U}$-frequently hypercyclic, Theorem \ref{anc_main_optimal1} ensures that the estimates we have proved are optimal when $\alpha\ne\frac{1}{\max(2,q)}$. 
\end{proof}

Now we deal with the case $p=1$.

\begin{theorem}\label{ufhcalpha_no_criticp1}
	Let $f$ be a $\mathcal{U}$-frequently hypercyclic function for the operator $T_{\alpha}$. Then, the following estimates hold
	\begin{align*}&\limsup_{r\rightarrow 1^{-}}\left((1-r)^{-\alpha}M_1(f,r)\right)=+\infty ,\quad\hbox{ if }\alpha\leq 0,\\
	&\limsup_{r\rightarrow 1^{-}}M_1(f,r)>0,\quad\hbox{ if }\alpha>0.
	\end{align*}
	For $\alpha\ne 0$, these results are optimal in the following sense: for any positive integer $l\geq 1$, there exists a $\mathcal{U}$-frequently hypercyclic function for the operator $T_{\alpha}$ such that for every $0<r<1$ 	sufficiently large $$M_1(f,r)\lesssim\left\{\begin{array}{ll} (1-r)^{\alpha}\log_l(-\log(1-r))& \hbox{ if } \alpha<0\\1&
	\hbox{ if } \alpha>0.\end{array}\right.$$
\end{theorem}

\begin{proof} First since a $\mathcal{U}$-frequently hypercyclic function is necessarily hypercyclic, the assertions (\ref{assa}) and (\ref{case2bhcalpha}) of Theorem \ref{hcalpha} ensures that, 
$$\limsup_{r\rightarrow 1^{-}}\left((1-r)^{-\alpha}M_1(f,r)\right)=+\infty,\hbox{ if }\alpha\leq 0,\quad\hbox{ and }\quad\limsup_{r\rightarrow 1^{-}}M_1(f,r)>0,\hbox{ if }\alpha>0.$$
Moreover, since a frequently hypercyclic function is necessarily $\mathcal{U}$-frequently hypercyclic, Theorem \ref{anc_main_optimal2} shows that the previous estimates are optimal when $\alpha\ne 0$.
\end{proof}

\section{Between $\mathcal{U}$-frequent hypercyclicity and hypercyclicity} \label{section_ubgamma} 
Let $1\leq p\leq\infty$. In view of Theorems \ref{anc_main_optimal1}, \ref{anc_main_optimal2}, \ref{ufhcalpha_no_critic} and \ref{ufhcalpha_no_criticp1}, the critical exponent related to the $L^p$ growth of frequently hypercyclic functions for $T_{\alpha}$ is the same as that related to the $L^p$ growth of $\mathcal{U}$-frequently hypercyclic functions. It is equal to $\frac{1}{\max(2,q)}$. Nevertheless this critical exponent is always equal to $0$ in the case of hypercyclic functions and, hence it does not depend on $p$. In this section we are interested in what happens between $\mathcal{U}$-frequent hypercyclicity and hypercyclicity. In particular, when $p>1$, we will try to understand why and how the critical exponent goes from $\frac{1}{\max(2,q)}$ in the case of $L^p$-norm of $\mathcal{U}$-frequent hypercyclic functions to $0$ for the $L^p$-norm of hypercyclic functions. To do this, we introduce intermediate notions of linear dynamics that link $\mathcal{U}$-frequent hypercyclicity and hypercyclicity. First of all, we need some definitions and results.

\subsection{Some weighted densities}\label{subsectiondens} 

First we introduce a refined notion of upper densities. 

%Let $\beta=(\beta_n)$ be a sequence of positive numbers satisfying the following conditions: 
%\begin{equation}\label{cond_seq}
%\beta\hbox{ is non-decreasing }\quad \hbox{ and }\quad\beta_n\rightarrow +\infty,\quad\hbox{ as }n\rightarrow +\infty.
%%\quad 
%%\left(\frac{\beta_n}{\sum_{j=1}^n\beta_j}\right)\rightarrow 0\quad\hbox{ as }n\rightarrow +\infty.
%\end{equation}
\begin{definition} Let $\beta=(\beta_n)$ be a non-decreasing sequence of positive real numbers tending to infinity. For a subset $E\subset \mathbb{N}$, its upper $\beta$-density is given by
	$$\overline{d}_{\beta}(E)=\limsup_{n\rightarrow +\infty}\frac{\sum_{k=1; k\in E}^n \beta_k}{\sum_{k=1}^n \beta_k}.$$ 
\end{definition}

These quantities enjoy all the classical properties of densities (see \cite{ErMo1,FrSa}) and allow to define dynamical notions of the same nature as hypercyclicity or $\mathcal {U}$-frequent hypercyclicity.

\begin{definition} Let $\beta=(\beta_n)$ be a non-decreasing sequence of positive real numbers tending to infinity and let $E$ be a subset of $\mathbb{N}$. An operator $T:X\rightarrow X$, where $X$ is a Fr\'echet space, is said to be $\mathcal{U}_{\beta}$-frequently hypercyclic if there exists $x\in X$ such that for every non-empty open subset $U\subset X$, 
	$$\overline{d}_{\beta}(\{n\in\mathbb{N}\ :\ T^nx\in U\})>0.$$
\end{definition}  

In the sequel, we are interested in densities given by the weighted sequence denoted by $\beta^{\gamma}$ and defined by $\beta^{\gamma}=(e^{n^{\gamma}})$, where $\gamma$ is a parameter with $0\leq \gamma\leq 1.$ First of all, let us notice that:
\begin{enumerate}[(i)] 
	\item the density $\overline{d}_{\beta^{0}}$ coincides with the upper natural density $\overline{d}$,
\item for any subset $E\subset\mathbb{N}$, $\overline{d}_{\beta^{1}}(E)>0$ if and only if $E$ is infinite.
\end{enumerate}

Moreover for $0<\gamma<1$ an integral comparison test leads to the estimate  
$$\displaystyle \sum_{k=1}^n e^{k^{\gamma}}\sim \frac{n^{1-\gamma}}{\gamma}e^{n^{\gamma}},\quad\hbox{ as }n\hbox{ tends to infinity,}$$
that we will use regularly in the rest of the paper. In addition, according to Lemma 2.8 of \cite{ErMo1} the following inequalities hold. 

\begin{lemma} For any $0\leq \gamma_1\leq \gamma_2\leq 1$ and for any subset $E$ of $\mathbb{N}$, we have
	$$\overline{d}(E)\leq \overline{d}_{\beta^{\gamma_1}}(E)\leq \overline{d}_{\beta^{\gamma_2}}(E)\leq \overline{d}_{\beta^{1}}(E).$$
\end{lemma}

Therefore the densities $\overline{d}_{\beta^{\gamma}}$ can give very different notions of dynamics that are intermediate between $\mathcal{U}$-frequent hypercyclicity and hypercyclicity. In particular the following lemma holds. 

\begin{lemma}\label{diffdensbeta} Let $0<\gamma\leq 1.$ There exists a subset $E_{\gamma}\subset\mathbb{N}$ such that, for any $0\leq \gamma'<\gamma\leq 1$, $\overline{d}_{\beta^{\gamma}}(E)>0$ and $\overline{d}_{\beta^{\gamma'}}(E)=0$.
\end{lemma}
\begin{proof} 
	First observe that, for all $0< t< 1$, 
	\begin{equation}\label{equ_lemme_dgamma}
\frac{\sum_{k=1}^{2^{n-1}}e^{k^{t}}}{\sum_{k=1}^{2^{n}}e^{k^{t}}}\sim 2^{-(1-t)}e^{-2^{nt}(1-2^{-t})}\rightarrow 0\hbox{ as }n\rightarrow +\infty.
	\end{equation}
	Let $\gamma\ne 1$. Set $E_{\gamma}=\mathbb{N}\bigcap\left(\bigcup_{n\geq \lfloor \frac{1}{\gamma}\rfloor +1}\left[2^{n}-\lfloor2^{n(1-\gamma)}\rfloor;2^n\right]\right)$. Clearly, for all $n$ large enough, we have 
	\begin{equation}\label{equ_lemme_dgamma2}
	\sum_{k=1;k\in E_{\gamma}}^{2^n}e^{k^{\gamma}}\geq  \sum_{k=2^n-\lfloor2^{n(1-\gamma)}\rfloor +1}^{2^{n}}e^{k^{\gamma}}.
	\end{equation}
	Moreover we get
	$$\frac{\sum_{k=2^n-\lfloor2^{n(1-\gamma)}\rfloor +1}^{2^{n}}e^{k^{\gamma}}}{\sum_{k=1}^{2^{n}}e^{k^{\gamma}}}=1-\frac{\sum_{k=1}^{2^n-\lfloor2^{n(1-\gamma)}\rfloor}e^{k^{\gamma}}}{\sum_{k=1}^{2^{n}}e^{k^{\gamma}}}.$$
	But we compute
	$$\frac{\sum_{k=1}^{2^n-\lfloor2^{n(1-\gamma)}\rfloor}e^{k^{\gamma}}}{\sum_{k=1}^{2^{n}}e^{k^{\gamma}}}\sim \left(1-\frac{\lfloor 2^{n(1-\gamma)}\rfloor}{2^n}\right)^{1-\gamma}e^{2^{n\gamma}((1-2^{-n}\lfloor 2^{n(1-\gamma)}\rfloor)^{\gamma}-1)}\rightarrow e^{-\gamma}\hbox{ as }n\rightarrow +\infty.$$
	Taking into account (\ref{equ_lemme_dgamma}) and (\ref{equ_lemme_dgamma2}), we deduce 
	$$\overline{d}_{\beta^{\gamma}}(E_{\gamma})>0.$$ 
	Now let $0\leq \gamma'<\gamma\leq 1$. Clearly keeping in mind that
		$$\sum_{k=1;k\in E_{\gamma}}^{2^n}e^{k^{\gamma'}}\leq \sum_{k=1}^{2^{n-1}}e^{k^{\gamma'}} + \sum_{k=2^n-\lfloor2^{n(1-\gamma)}\rfloor}^{2^{n}}e^{k^{\gamma'}}$$ 
	and by using both $\gamma'-\gamma<0$, the estimate
	$$\frac{\sum_{k=1}^{2^n-\lfloor2^{n(1-\gamma)}\rfloor}e^{k^{\gamma'}}}{\sum_{k=1}^{2^{n}}e^{k^{\gamma'}}}\sim \left(1-\frac{\lfloor 2^{n(1-\gamma)}\rfloor}{2^n}\right)^{1-\gamma'}e^{2^{n\gamma'}((1-2^{-n}\lfloor 2^{n(1-\gamma)}\rfloor)^{\gamma'}-1)}\rightarrow 1\hbox{ as }n\rightarrow +\infty,$$
	and (\ref{equ_lemme_dgamma}) we derive
	$$\overline{d}_{\beta^{\gamma'}}(E_{\gamma})=0.$$
	Finally, for $\gamma=1$, the lemma is easy to establish since a subset $E\subset\mathbb{N}$ satisfies $d_{\beta^1}(E)>0$ if and only if $E$ is infinite. This finishes the proof.
	\end{proof}

In some sense, the densities $\underline{d}_{\beta^\gamma}$, $0\leq\gamma\leq 1$, will us allow to interpolate the behavior of hypercyclic vectors between the $\mathcal{U}$-frequent hypercyclicity and the hypercyclicity.

\subsection{Rate of growth of $\mathcal{U}_{\beta^{\gamma}}$-frequently hypercyclic functions }\label{subsectiondirect}
First we deal with the case $p>1$. We will discuss the case $p=1$ at the end of the section. We are ready to state the result that highlights both the continuous variation of the critical exponent and that of the growth (in term of $L^p$ averages) of a hypercyclic function for $T_{\alpha}$ according to the frequency of visits of non-empty open subsets by its orbit under the action of $T_{\alpha}$.

\begin{theorem}\label{ubetafhcalpha_no_critic}
	Let $0<\gamma<1$ and $1< p\leq\infty$. Let $f$ be a $\mathcal{U}_{\beta^{\gamma}}$-frequently hypercyclic function for the operator $T_{\alpha}$. Then, the following hold  
		\begin{align*}
%			&\hbox{ if } \alpha\leq 0,\quad \limsup_{r\rightarrow 1}\left(\left[1-r\right]^{(1-\gamma)(\frac{1}{\max(2,q)}-\alpha)}M_p(f,r)\right)>0,\\
			&\hbox{ if } \alpha<\frac{1-\gamma}{\max(2,q)},\quad \limsup_{r\rightarrow 1^{-}}\left(\left[1-r\right]^{\frac{1-\gamma}{\max(2,q)}-\alpha}M_p(f,r)\right)>0,\\
			&\hbox{ if } \alpha=\frac{1-\gamma}{\max(2,q)} ,\quad \limsup_{r\rightarrow 1^{-}}\left(M_p(f,r)\right)=+\infty,\\
			&\hbox{ if } \alpha>\frac{1-\gamma}{\max(2,q)} ,\quad \limsup_{r\rightarrow 1^{-}}\left(M_p(f,r)\right)>0.
			%\hbox{ if } \alpha>\frac{1}{\max(2,q)}
		\end{align*}
\end{theorem}

\begin{proof}
	Let $f$ be a $\mathcal{U}_{\beta^{\gamma}}$-frequently hypercyclic function for $T_{\alpha}$. We write $f=\sum\limits_{k\geq 0}\frac{a_k}{(k+1)^{\alpha}}z^k.$ Since $f$ is $\mathcal{U}_{\beta^{\gamma}}$-frequently hypercyclic there exists an increasing sub-sequence $(n_k)\subset\mathbb{N}$ with positive upper $\beta^{\gamma}$-density such that, for all $k\geq 1$, 
	$$\vert T_{\alpha}^{n_k}f(0)-3/2\vert=\vert a_{n_k}-3/2\vert <1/2.$$
	We get, for all $k\geq 1$, $\vert a_{n_k}\vert \geq 1$. Set $I=((n_k))$ and for all $N\geq 1$, $I_N=I\cap\{1,\dots,N\}$. The hypothesis $\overline{d}_{\beta^{\gamma}}(I)>0$ ensures that there exist $0<C<1$ and an increasing sequence $(N_l)$ of positive integers such that
	\begin{equation}\label{dominationlimsup}
		\sum_{k\in I_{N_l}}e^{k^{\gamma}} \geq  C  \frac{N_l^{1-\gamma}}{\gamma}e^{N_l^{\gamma}}.
	\end{equation}
Up to take a sub-sequence, we can suppose that
		\begin{equation}\label{equlimsup100}
	C(N_{k+1}+1)\geq N_k+1.
	\end{equation}
	Let us consider, for all $l\geq 1$, a sequence $(r_l)$ with $1-\frac{1}{N_l-1}\leq r_l<1-\frac{1}{N_l}$. Observe that
	\begin{equation}\label{encNl2}
		N_l-1\leq \frac{1}{1-r_l}<N_l.
	\end{equation}
	\begin{enumerate}
		\item \textbf{Case $2\leq p\leq\infty$:}\\
		Jensen'inequality and Parseval's Theorem give
		\begin{equation}\label{eq_princ_again00}
		[M_p(f,r_l)]^2\geq [M_2(f,r_l)]^2=\sum_{k\geq 0}\frac{\vert a_k\vert^2}{(k+1)^{2\alpha}}r_l^{2k}
		\geq \sum_{k=1}^{N_l}\frac{\vert a_k\vert^2}{(k+1)^{2\alpha}}r_l^{2k}\gtrsim  \sum_{k=1}^{N_l}\frac{\vert a_k\vert^2}{(k+1)^{2\alpha}}.
		\end{equation}
		Let us choose $j_0\in\mathbb{N}$ such that the function $t\mapsto t^{-2\alpha}e^{-t^{\gamma}}$ is decreasing for $t\geq N_{j_0}$. 
Thus we can write, for all $l\geq j_0+1$, 
$$[M_p(f,r_l)]^2\gtrsim \sum_{j=1+j_0}^l\sum_{k=1+N_{j-1}}^{N_j}\frac{\vert a_k\vert^2}{(k+1)^{2\alpha}}e^{k^{\gamma}}e^{-k^{\gamma}}.$$
Then applying Lemma \ref{Abel_Transform} with $u_k=\vert a_k\vert^2 e^{k^{\gamma}}$, we get 

\begin{equation}\label{equ-principal_Abel}
\begin{array}{ll}\displaystyle[M_p(f,r_l)]^2\gtrsim	&\displaystyle S_{N_l}(N_l+1)^{-2\alpha}e^{-N_{l}^{\gamma}}-\displaystyle S_{N_{j_0}}(N_{j_0}+1)^{-2\alpha}e^{-N_{j_0}^{\gamma}}\\&+\displaystyle\sum_{j=1+j_0}^l S_{N_{j-1}}\left((N_{j-1}+1)^{-2\alpha}e^{-N_{j-1}^{\gamma}}-(N_{j}+1)^{-2\alpha}e^{-N_{j}^{\gamma}}\right).\end{array}
\end{equation}

Since $S_{N_i}=\sum\limits_{k\leq N_i}\vert a_k\vert^2 e^{k^{\gamma}}$, by construction and by (\ref{dominationlimsup}), we get, for all $i\geq 1$,

\begin{equation}\label{equ-principal_Abel2}
S_{N_i}\geq \sum_{k\in I_{N_i}}e^{k^{\gamma}}\gtrsim N_{i}^{1-\gamma}e^{N_{i}^{\gamma}}\gtrsim (N_{i}^{1-\gamma}+1)e^{N_{i}^{\gamma}}.
\end{equation}

From (\ref{equ-principal_Abel}) and (\ref{equ-principal_Abel2}) we deduce

\begin{equation}\label{equ-principal_Abel3}
\begin{array}{ll}\displaystyle[M_p(f,r_l)]^2\gtrsim	&\displaystyle (N_l+1)^{(1-\gamma)-2\alpha}\\&+\displaystyle\sum_{j=1+j_0}^l (N_{j-1}+1)^{1-\gamma}e^{N_{j-1}^{\gamma}}\left((N_{j-1}+1)^{-2\alpha}e^{-N_{j-1}^{\gamma}}-(N_{j}+1)^{-2\alpha}e^{-N_{j}^{\gamma}}\right).\end{array}
\end{equation}

		\begin{enumerate}
			\item \textit{Case $\alpha < \frac{1-\gamma}{2}$}: 	
			From (\ref{equ-principal_Abel3}) we get, for $l$ large enough
			\begin{equation}\label{comb2}[M_p(f,r_l)]^2\gtrsim	\displaystyle (N_l+1)^{(1-\gamma)-2\alpha}
			\end{equation}
			Thanks to (\ref{encNl2}) and (\ref{comb2}) we deduce
			$$[M_p(f,r_l)]^2\gtrsim (1-r_l)^{2\alpha-(1-\gamma)}.$$
			Hence we conclude 
			$$\limsup_{r\rightarrow 1^{-}}\left[(1-r)^{-\alpha+\frac{1-\gamma}{2}}M_p(f,r)\right]>0.$$
			
			\item \textit{Case $\alpha = \frac{1-\gamma}{2}$}: taking into consideration (\ref{equ-principal_Abel3}), we can write, for all $l\geq 1+j_0$,
			$$\displaystyle[M_p(f,r_l)]^2\gtrsim\displaystyle \sum_{j=1}^l \left(1-\left(\frac{N_{j-1}+1}{N_{j}+1}\right)^{1-\gamma}e^{N_{j-1}^{\gamma}-N_{j}^{\gamma}}\right).$$
			Thus taking into account (\ref{equlimsup100}) we derive, for all $l\geq 1+j_0$, $[M_p(f,r_l)]^2\gtrsim l$, which allows to obtain $$\displaystyle\limsup_{r\rightarrow 1^{-}}M_p(f,r)=+\infty.$$
			
			\item \textit{Case $\alpha > \frac{1-\gamma}{2}$}: since $f$ is hypercyclic, the conclusion is given by Theorem \ref{hcalpha}.
		\end{enumerate}

		\item \textbf{Case $1< p < 2$}:\\
		It suffices to combine the arguments of the proof of the preceding case with those of the proof of (\ref{casnonrefait1}) of Theorem \ref{ufhcalpha_no_critic} to obtain the desired conclusions.
		\end{enumerate}	
\end{proof}

\subsection{Optimal growth of $\mathcal{U}_{\beta^{\gamma}}$-frequently hypercyclic functions: a constructive proof}\label{subsectionconstructive}
In this subsection, we intend to prove that the estimates given by Theorem \ref{ubetafhcalpha_no_critic} whenever $\alpha$ is different from the critical exponent, i.e. $\alpha\ne \frac{1-\gamma}{\max(2,q)}$, are optimal. The case $\alpha = \frac{1-\gamma}{\max(2,q)}$ will be treated separately in Section \ref{sectioncritic}. Thus for all $0<\gamma<1$ and for $\alpha\ne \frac{1-\gamma}{\max(2,q)}$, we propose to build $\mathcal{U}_{\beta^{\gamma}}$-frequently hypercyclic functions for $T_{\alpha}$ that have the required $L^p$ growth and no more. To do this, we follow the construction of frequently hypercyclic functions for $T_{\alpha}$ given in \cite{MouMun3} which itself was partly inspired by \cite{Drasin}. In particular we will need the so-called Rudin-Shapiro polynomials (combined with the de la Vall\'{e}e-Poussin polynomials), which have coefficients 
$\pm 1$ (or bounded by $1$) and an optimal growth of $L^p$-norm. Let us recall the associated result in the form of Lemma 2.1 of \cite{Drasin} that summarized the result of Rudin-Shapiro \cite{RudSha}.  

\begin{lemma}\label{lemma_rud_shap} 
	\begin{enumerate}
		\item For each $N\geq 1$, there is a trigonometric polynomial $p_N=\sum_{k=0}^{N-1}\varepsilon_{N,k}e^{ik\theta}$ 
		where $\varepsilon_{N,k}=\pm 1$ for all $0\leq k\leq N-1$ with at least half of the coefficients being $+1$ and with 
		$$\Vert p_{N}\Vert_{p}\leq 5\sqrt{N}\hbox{ for }p\in [2,+\infty].$$
		\item For each $N\geq 1$, there is a trigonometric polynomial $p_N^*=\sum_{k=0}^{N-1}a_{N,k}e^{ik\theta}$ 
		where $\vert a_{N,k}\vert\leq 1$ for all $0\leq k\leq N-1$ with at least $\lfloor \frac{N}{4} \rfloor$ coefficients being $+1$ and with 
		$$\|p^{*}_{N}\|_{p}\leq 3N^{1/q}\hbox{ for }p\in [1,2].$$
	\end{enumerate} 
\end{lemma}

For any given polynomial $q$ with $q(z)=\sum_{j=0}^d b_j z^j$ with $b_d\ne 0$, we denote 
$d=\hbox{deg}(q)$ and $\displaystyle \Vert q\Vert_{\ell^1}=\sum_{j=0}^d\vert b_j\vert$. 
We set $\displaystyle 2\mathbb{N}=\bigcup_{k\geq 1}\mathcal{A}_{k}$ where for any $k\geq 1,$ 
$\displaystyle \mathcal{A}_{k}=\left\{2^{k}(2j-1);j\in\mathbb{N} \right\}.$ Denote by $\mathcal{P}$ the countable set of polynomials 
with rational coefficients and let us also consider 
pairs $(q,l)$ with $q\in\mathcal{P}$ and $l\in \mathbb{N}$ satisfying $\Vert q\Vert_{\ell^1}\leq l$, displayed as a single sequence $(q_k,l_k)$. 
%Thus the following property holds: $\Vert q_k\Vert_{\ell^1}\leq l_k$. Let us consider an enumeration $(q_k)$ of 
%$\mathcal{P}$ and a sequence $(l_k)$ tending to $+\infty$ such that $\Vert q_k\Vert_{\ell^1}\leq l_k$. 
Clearly $(q_k)$ is a dense set in $H(\mathbb{D})$. 
Hence, for any $k\geq 1,$ we set $d_k=\deg (q_k)$ and we have $$\Vert q_k\Vert_{\ell^1}\leq l_k\hbox{ for every }k\geq 1.$$

For any $\alpha\in\mathbb{R},$ for any positive integer 
$k\geq 1,$ we set $\displaystyle \tilde{q_{k}}(z)=\sum_{j=0}^{d_{k}}(j+1)^\alpha b_{j}^{(k)}z^{j}$ 

Let $\alpha$ be a real number and $p\in (1,\infty].$ For all integer $n\geq 0,$ 
we set $I_{n}=\{2^{n},\ldots,2^{n+1}-1\}.$ Next, for $k\geq 1,$ let us define the integers 
$$\alpha_{k}=1+
\left\lfloor \max\left(l_{k}^{2}(1+d_{k})^{2\max(\alpha, 0)},d_{k}+
\max(3,3+\alpha)l_{k}^2+\max(\alpha, 0)l_{k}\log(1+d_{k})\right)\right\rfloor$$
and 
$$\alpha^{*}_{k}=1+
\left\lfloor \max\left(l_{k}^{q}(1+d_{k})^{q\max(\alpha, 0)},d_{k}+\max(3,3+\alpha)l_{k}^2+\max(\alpha, 0)l_{k}\log(1+d_{k})
\right)\right\rfloor.$$ 

We set $\displaystyle f_{\alpha}=\sum\limits_{n\geq 0}P_{n,\alpha}$ where the blocks $(P_{n,\alpha})$ are polynomials 
defined as follows, using Rudin-Shapiro polynomials given by Lemma \ref{lemma_rud_shap}, 
\begin{equation}\label{poly_2inf}
P_{n,\alpha}(z)  =  \left\{\begin{array}{l} 
0 \hbox{ if } n \mbox{ is odd }\\
0 \hbox{ if } n\in \mathcal{A}_{k}\hbox{ and } 2^{n-1}<\alpha_{k}\\
z^{2^{n}}Q_{n}(z) \hbox{ if } n\in \mathcal{A}_{k}\hbox{ and }2^{n-1}\geq \alpha_{k}
\end{array}\right.
\end{equation}
with for $n\in \mathcal{A}_{k}$, 
$$Q_{n}(z)=\sum\limits_{j\in I_{n}}(j+1)^{-\alpha}c_{j-2^{n}}^{(k)}z^{j-2^{n}}$$ 
where the sequence $(c_{j}^{(k)})$ denotes the sequence of 
the coefficients of the polynomial  
$\displaystyle p_{\lfloor \frac{2^{n(1-\gamma)}}{\alpha_{k}} \rfloor}(z^{\alpha_k})\tilde{q_{k}}(z)$. 

We also set $\displaystyle f^{*}_{\alpha}=\sum\limits_{n\geq 0}P^{*}_{n,\alpha}$ where the blocks $(P^{*}_{n,\alpha})$ are polynomials 
defined as follows, using the de la Vall\'{e}e-Poussin polynomials given by Lemma \ref{lemma_rud_shap},  
\begin{equation}\label{poly_12}
P^{*}_{n,\alpha}(z) = \left\{\begin{array}{l} 
0 \hbox{ if } n \hbox{ is odd }\\
0 \hbox{ if } n\in \mathcal{A}_{k}\hbox{ and } 2^{n-1}<\alpha^{*}_{k}\\
\displaystyle z^{2^{n}}Q_{n}^{*}(z)\hbox{ if } n\in \mathcal{A}_{k}\hbox{ and } 2^{n-1}\geq\alpha^{*}_{k},
\end{array}\right.
\end{equation}
with, for $n\in \mathcal{A}_{k}$,
$$Q_{n}^{*}(z)=\sum\limits_{j\in I_{n}}(j+1)^{-\alpha}c_{j-2^{n}}^{(k)}z^{j-2^{n}}$$ 
where the sequence $(c_{j}^{(k)})$ denotes the sequence of 
the coefficients of the polynomial $\displaystyle p^{*}_{\lfloor \frac{2^{n(1-\gamma)}}{\alpha^{*}_{k}} \rfloor}(z^{\alpha^{*}_{k}})\tilde{q_{k}}(z)$.
\vskip2mm

A combination of Lemma \ref{Conv1} below with the triangle inequality shows 
that the function $f_{\alpha}$ (resp. $f_{\alpha}^*$) belongs 
to $H(\mathbb{D})$. Observe that, if we denote the polynomial $z\mapsto p_{\lfloor \frac{2^{n-1}}{\alpha_{k}} \rfloor}(z^{\alpha_k})$ 
(resp. $z\mapsto p_{\lfloor \frac{2^{n-1}}{\alpha_{k}} \rfloor}^*(z^{\alpha_k^*})$) by $g_k$ (resp. $g_k^*$), we have, 
for all $1\leq p\leq +\infty,$ $\Vert g_k \Vert_p=
\Vert p_{\lfloor \frac{2^{n-1}}{\alpha_{k}} \rfloor}\Vert_p$ (resp. 
$\Vert g_k^* \Vert_p=\Vert p_{\lfloor \frac{2^{n-1}}{\alpha_{k}^*} \rfloor}^*\Vert_p$). 
Finally for any integer $n$, let us denote $(\phi_n(k))$ the sequence defined as follows
$$\phi_{n}(k)=\left\{\begin{array}{ll}(k+1)^{-\alpha}&\hbox{ if }k\in I_n\\0&\hbox{ otherwise.}
\end{array}\right.$$

\begin{lemma}\label{Conv1} Let $\alpha\in\mathbb{R}$. The following estimates hold. 
	\begin{enumerate}[(i)]
		\item For any $2\leq p\leq +\infty,$ any $0<r<1$ and any $n\in\mathbb{N},$ 
		we have $$M_p(P_{n,\alpha},r)\lesssim 2^{n(\frac{1-\gamma}{2}-\alpha)}r^{2^{n}}.$$
		\item For any $1<p<2,$  any $0<r<1$ and any $n\in\mathbb{N},$ 
		we have $$M_p(P^{*}_{n,\alpha},r)\lesssim 2^{n(\frac{1-\gamma}{q}-\alpha)}r^{2^{n}}.$$
		%\item For any $n\in\mathbb{N}$, we have 
		%$$M_{\infty}(P_{n,\alpha},r)\lesssim 2^{n(\frac{1-\gamma}{2}-\alpha)}r^{2^{n}}$$ 
	\end{enumerate}
\end{lemma}

\begin{proof}  \begin{enumerate}[(i)]
		\item On one hand, we deal with the case $2\leq p<+\infty.$ Let $n$ be a positive integer. Without loss of generality, we can assume that $n$ belongs to 
		the set $\mathcal{A}_{k}$ for some $k\geq 1$. Let $r$ be in $(0,1).$ Since $\displaystyle r \mapsto M_p(f,.)$ is increasing, we get  
		$$M_p(P_{n,\alpha},r)\leq r^{2^{n}}\|Q_{n}\|_{p}.$$ 
		Then, the polynomial $Q_{n}$ can be viewed as a trigonometric polynomial obtained by an abstract convolution operator on 
		$\mathbb{T},$ given by $(c_{k})_{k\geq 0} \mapsto (\phi_{n}(j)c_{j-2^{n}}^{(k)})_{j\geq 0}$ 
		(where $(c_{j}^{(k)})$ denotes the sequence of the coefficients of the polynomial $\displaystyle p_{\lfloor \frac{2^{n-1}}{\alpha_{k}} \rfloor}\tilde{q_{k}}$). 
		Now, we are going to apply the Marcinkiewicz Multiplier Theorem \cite[Theorem 8.2 p.148]{EdwGau}. 
		To do this, observe that we have, for any $l\geq 1,$ 
		$$\sup_{j\in I_{l}}\vert \phi_{n}(j) \vert \leq\sup_{j\in I_{n}}\vert \phi_{n}(j) \vert \lesssim 2^{-n\alpha}$$ 
		and
		$$\sup_{l}\sum_{j\in I_{l}}\vert \phi_{n}(j+1)-\phi_{n}(j)\vert \leq \sum_{j\in I_{n}}\vert \phi_{n}(j+1)-\phi_{n}(j)\vert \lesssim 2^{-n\alpha}.$$ 
		Hence, taking into account the choice of $\alpha_k$ and Lemma \ref{lemma_rud_shap}, we get
		$$\begin{array}{rl}
		\|Q_{n}\|_{p} & \lesssim 2^{-n\alpha} \Vert p_{\lfloor \frac{2^{n(1-\gamma)}}{\alpha_{k}}\rfloor}\Vert_{p} \|\tilde{q_{k}}\|_{\infty}\\
		& \lesssim 2^{-n\alpha} \sqrt{\frac{2^{n(1-\gamma)}}{\alpha_{k}}}\ l_{k}(1+d_{k})^{\max(\alpha,0)}\\
		& \lesssim 2^{n(\frac{1-\gamma}{2}-\alpha)}.\end{array}$$
		
		Finally, we obtain the desired estimate 
		$$M_p(P_{n,\alpha},r)\lesssim 2^{n(\frac{1-\gamma}{2}-\alpha)}r^{2^{n}}.$$ 
		\\
		On the other hand we deal with the case $p=\infty$. Let us recall that $P_{n,\alpha}(z)=0$ or $z^{2^{n}} Q_{n}(z)$ with $Q_{n}(z)=\sum\limits_{j\in I_{n}}(j+1)^{-\alpha}c_{j-2^{n}}^{(k)}z^{j-2^{n}}$ 
		where $(c_{j}^{(k)})$ denotes the sequence of 
		the coefficients of the polynomial $p_{\lfloor \frac{2^{n(1-\gamma)}}{\alpha_{k}} \rfloor}(z^{\alpha_k})\tilde{q_{k}}(z)$. First, assume that $\alpha\leq 0$. We write
		$$ M_{\infty}(P_{n,\alpha},r)\lesssim r^{2^n} \Vert Q_{n} \Vert_{\infty}.$$ 
		Using the form of $Q_n$, as in the proof of Lemma 3.6 of \cite{MouMun3} we apply a fractional Bernstein's inequality to obtain, taking into consideration Lemma \ref{lemma_rud_shap},
		$$ M_{\infty}(P_{n,\alpha},r)\lesssim r^{2^n} 2^{-n\alpha}\Vert Q_{n} \Vert_{\infty}
		\lesssim 2^{-n\alpha}\Vert p_{\lfloor \frac{2^{n(1-\gamma)}}{\alpha_{k}} \rfloor}\Vert_{\infty}
		\Vert \tilde{q_{k}} \Vert_{\infty}\lesssim 2^{-n\alpha}\sqrt{\lfloor \frac{2^{n(1-\gamma)}}{\alpha_{k}} \rfloor} l_k.$$ 
		Thanks to the choice of $\alpha_k$, we have, for $\alpha\leq 0$, 
		$$M_{\infty}(P_{n,\alpha},r)\lesssim 2^{n(\frac{1-\gamma}{2}-\alpha)}.$$
		To conclude it suffices to mimic the induction of the proof of Lemma 3.7 of \cite{MouMun3} 
		\item The proof is similar as that of the case $2\leq p<+\infty$ by applying Lemma \ref{lemma_rud_shap} for $1<p<2$. 
	\end{enumerate}   
\end{proof}

%\begin{lemma} \label{Conv2}
%For any $1<p<2,$  any $0<r<1$ and any $n\in\mathbb{N},$ 
%we have $$M_p(r,P^{*}_{n,\alpha})\lesssim 2^{n(\frac{1-\gamma}{q}-\alpha)}r^{2^{n}}.$$
%\end{lemma}

Now we are ready to obtain the rate of growth of the aforementioned functions $f_{\alpha}$ and $f^*_{\alpha}.$ We refer to Lemmas 3.4, 3.5 and 3.8 of \cite{MouMun3} with obvious modifications. 

\begin{lemma}\label{estim1falpha} \begin{enumerate}
		\item Let $2\leq p\leq+\infty.$ For all $0<r<1,$ the following estimates hold
		$$M_p(f_{\alpha},r)\lesssim \left\{\begin{array}{l}(1-r)^{\alpha-\frac{1-\gamma}{2}}\hbox{ if }\alpha<\frac{1-\gamma}{2},\\
		%\sqrt{\vert\log (1-r)\vert}\hbox{ if }\alpha=1/2,\\
		1\hbox{ if }\alpha>\frac{1-\gamma}{2}\end{array}\right.$$
		
		\item Let $1<p<2$. For all $0<r<1,$ the following estimates hold
		$$M_p(f_{\alpha}^{*},r)\lesssim\left\{\begin{array}{l}(1-r)^{\alpha-\frac{1-\gamma}{q}}\hbox{ if }\alpha<\frac{1-\gamma}{q},\\
		%\vert\log (1-r)\vert^{\frac{1}{p}}\hbox{ if }\alpha=1/p',\\
		1\hbox{ if }\alpha>\frac{1-\gamma}{q}\end{array}\right.$$
	\end{enumerate}
\end{lemma}

Now we are going to prove that the functions $f_{\alpha}$ and $f^{*}_{\alpha}$ are $U_{\beta^{\gamma}}$-frequently hypercyclic for $T_{\alpha}$.

\begin{proposition}\label{propfhcdb} For $p\geq 2$ (resp. $1<p<2$), the function $f_{\alpha}$ (resp. $f^{*}_{\alpha}$) 
	is a $U_{\beta^{\gamma}}$-frequently hypercyclic vector for the operator $T_{\alpha}$.
\end{proposition}

\begin{proof}
	We only prove the frequent hypercyclicity of $f_{\alpha}$ for the operator $T_{\alpha}$. We do not repeat the 
	details for $f^{*}_{\alpha}$: it will be enough to make the appropriate modifications. 
	\vskip2mm
	
	Let $k$ be a large enough integer. Let us consider $n\in\mathcal{A}_{k}$ such that $2^{n-1}\geq \alpha_{k}.$ We consider $\mathcal{B}_{n}$ the set of $s$ in 
	$I_n$ such that the coefficient $z^{s}$ in the polynomial $z^{2^{n}}p_{\lfloor \frac{2^{n(1-\gamma)}}{\alpha_{k}} \rfloor}(z^{\alpha_{k}})$ is equal to $1$ and we denote by $\displaystyle T_{k}=\left\{s: s\in \mathcal{B}_{n},\ n\in\mathcal{A}_{k}, \  2^{n-1}\geq \alpha_{k}\right\}.$ \\
	Observe that $\max(\mathcal{B}_{n})\leq 2^n+\lfloor2^{n(1-\gamma)}\rfloor$ and since at least half of the coefficients of $p_{\lfloor \frac{2^{n(1-\gamma)}}{\alpha_{k}} \rfloor}$ being $+1$, we get 
		\begin{equation}\label{denshdcons}\frac{\sum\limits_{j\leq \max(\mathcal{B}_{n});\atop j\in T_{k}}e^{j^{\gamma}}}{\sum\limits_{j\leq \max(\mathcal{B}_{n})}e^{j^{\gamma}}}\geq 
			\frac{\sum\limits_{j=2^n}^{2^n+\lfloor 2^{-1}\lfloor\frac{2^{n(1-\gamma)}}{\alpha_k}-1\rfloor\rfloor}e^{j^{\gamma}}}{\sum\limits_{j=1}^{2^n+\lfloor2^{n(1-\gamma)}\rfloor}e^{j^{\gamma}}}=
			\frac{\sum\limits_{j=1}^{2^n+\lfloor 2^{-1}\lfloor\frac{2^{n(1-\gamma)}}{\alpha_k}-1\rfloor\rfloor}e^{j^{\gamma}}}{\sum\limits_{j=1}^{2^n+\lfloor2^{n(1-\gamma)}\rfloor}e^{j^{\gamma}}}-\frac{\sum\limits_{j=1}^{2^n-1}e^{j^{\gamma}}}{\sum\limits_{j=1}^{2^n+\lfloor2^{n(1-\gamma)}\rfloor}e^{j^{\gamma}}}..
	\end{equation}
%	\begin{equation}\label{denshdcons}\frac{\sum\limits_{j\leq \max(\mathcal{B}_{n});\atop j\in T_{k}}e^{j^{\gamma}}}{\sum\limits_{j\leq \max(\mathcal{B}_{n})}e^{j^{\gamma}}}\geq 
%	\frac{\sum\limits_{j=2^n}^{2^n+2^{-1}\lfloor2^{n(1-\gamma)}\rfloor}e^{j^{\gamma}}}{\sum\limits_{j=1}^{2^n+\lfloor2^{n(1-\gamma)}\rfloor}e^{j^{\gamma}}}=
%	\frac{\sum\limits_{j=1}^{2^n+2^{-1}\lfloor2^{n(1-\gamma)}\rfloor}e^{j^{\gamma}}}{\sum\limits_{j=1}^{2^n+\lfloor2^{n(1-\gamma)}\rfloor}e^{j^{\gamma}}}-\frac{\sum\limits_{j=1}^{2^n-1}e^{j^{\gamma}}}{\sum\limits_{j=1}^{2^n+\lfloor2^{n(1-\gamma)}\rfloor}e^{j^{\gamma}}}.
%	\end{equation}

	Clearly we have
$$\left(2^n+\lfloor 2^{-1}\lfloor\alpha_k^{-1}2^{n(1-\gamma)}-1\rfloor\rfloor\right)^{\gamma}-\left(2^n+\lfloor2^{n(1-\gamma)}\rfloor\right)^{\gamma}\rightarrow -\gamma(1-\frac{1}{2\alpha_k})$$
and
$$\left(2^n-1\right)^{\gamma}-\left(2^n+\lfloor2^{n(1-\gamma)}\rfloor\right)^{\gamma}\rightarrow -\gamma,$$
which implies, using similar estimations as those of the proof of Lemma \ref{diffdensbeta},
$$\frac{\sum\limits_{j=1}^{2^n+\lfloor 2^{-1}\lfloor\frac{2^{n(1-\gamma)}}{\alpha_k}-1\rfloor\rfloor}e^{j^{\gamma}}}{\sum\limits_{j=1}^{2^n+\lfloor2^{n(1-\gamma)}\rfloor}e^{j^{\gamma}}}\rightarrow e^{-\gamma(1-\frac{1}{2\alpha_k})}\quad 
\hbox{ and }\quad\frac{\sum\limits_{j=1}^{2^n-1}e^{j^{\gamma}}}{\sum\limits_{j=1}^{2^n+\lfloor2^{n(1-\gamma)}\rfloor}e^{j^{\gamma}}}\rightarrow e^{-\gamma},\quad \hbox{ as }n\rightarrow +\infty.$$
	
%	Clearly we have
%	$$\left(2^n+\frac{\lfloor2^{n(1-\gamma)}\rfloor}{2}\right)^{\gamma}-\left(2^n+\lfloor2^{n(1-\gamma)}\rfloor\right)^{\gamma}\rightarrow -\frac{\gamma}{2}$$
%	and
%	$$\left(2^n-1\right)^{\gamma}-\left(2^n+\lfloor2^{n(1-\gamma)}\rfloor\right)^{\gamma}\rightarrow -\gamma,$$
%	which implies, using similar estimations as those of the proof of Lemma \ref{diffdensbeta}, 
%	$$\frac{\sum\limits_{j=1}^{2^n+2^{-1}\lfloor2^{n(1-\gamma)}\rfloor}e^{j^{\gamma}}}{\sum\limits_{j=1}^{2^n+\lfloor2^{n(1-\gamma)}\rfloor}e^{j^{\gamma}}}\rightarrow e^{-\gamma/2}\quad 
%	\hbox{ and }\quad\frac{\sum\limits_{j=1}^{2^n-1}e^{j^{\gamma}}}{\sum\limits_{j=1}^{2^n+\lfloor2^{n(1-\gamma)}\rfloor}e^{j^{\gamma}}}\rightarrow e^{-\gamma},\quad \hbox{ as }n\rightarrow +\infty.$$
	Hence the inequality (\ref{denshdcons}) ensures that 
	$$\overline{d}_{\beta^{\gamma}}(T_k)>0.$$
	Then let $\alpha$ be a real number and let $k\in\mathbb{N}$. Let us consider $s\in\mathcal{B}_{n}$ with $n\in\mathcal{A}_{k}$ satisfying 
	$2^{n-1}\geq\alpha_{k}.$ As in the proof of Lemma 3.9 of \cite{MouMun3} with easy modifications, we can prove that
	$$\sup_{\vert z\vert = 1-\frac{1}{l_{k}}} \vert T_{\alpha}^{s}(f_{\alpha})(z)-q_{k}(z)\vert \lesssim \frac{1}{l_{k}},$$ 
	provided that $k$ is chosen large enough. This allows to obtain the frequent hypercyclicity of $f_{\alpha}$.
\end{proof}

In summary, Lemma \ref{estim1falpha} and Proposition \ref{propfhcdb} leads to the following result, which shows that the statement of Theorem \ref{ubetafhcalpha_no_critic} is optimal whenever $\alpha$ is not the critical exponent.

\begin{theorem}\label{thmubetaopti} Let $0<\gamma<1$ and $1<p\leq\infty$. 
	\begin{enumerate} 
		\item for $\alpha<\frac{1-\gamma}{\max(2,q)}$ there exists a $\mathcal{U}_{\beta^{\gamma}}$-frequently hypercyclic function for the operator $T_{\alpha}$ such that 
		$$M_p(f,r)\lesssim (1-r)^{\alpha-\frac{1-\gamma}{\max(2,q)}};$$
	\item for $\alpha>\frac{1-\gamma}{\max(2,q)}$ there exists a $\mathcal{U}_{\beta^{\gamma}}$-frequently hypercyclic function for the operator $T_{\alpha}$ such that 
	$$M_p(f,r)\lesssim 1.$$
	\end{enumerate}	
\end{theorem}

\begin{remark}
	{\rm Let $0<\gamma<1$. It seems important to note that the functions constructed for the proof of Theorem \ref{thmubetaopti} are $\mathcal{U}_{\beta^{\gamma}}$-frequently hypercyclic for $T_{\alpha}$ but neither $\mathcal{U}_{\beta^{\gamma'}}$-frequently hypercyclic for $0 < \gamma'<\gamma$ nor $\mathcal{U}$-frequently hypercyclic, since they don't satisfy the estimates given by Theorem \ref{ubetafhcalpha_no_critic} or Theorem \ref{ufhcalpha_no_critic}.}
\end{remark}

Finally let us say some words for the case $p=1$. As in the proof of Theorem \ref{ufhcalpha_no_criticp1}, observe that a $\mathcal{U}_{\beta^{\gamma}}$-frequently hypercyclic function is necessarily hypercyclic and a $\mathcal{U}$-frequently hypercyclic function is necessarily $\mathcal{U}_{\beta^{\gamma}}$-frequently hypercyclic. This leads to the following statement. 

\begin{theorem}\label{udfhcalpha_no_criticp1}
Let $0<\gamma<1$. Let $f$ be a $\mathcal{U}_{\beta^{\gamma}}$-frequently hypercyclic function for the operator $T_{\alpha}$. Then, the following assertions hold 
	\begin{align*}&\limsup_{r\rightarrow 1^{-}}\left((1-r)^{-\alpha}M_1(f,r)\right)=+\infty ,\quad\hbox{ if }\alpha\leq 0,\\
	&\limsup_{r\rightarrow 1^{-}}M_1(f,r)>0,\quad\hbox{ if }\alpha>0.
	\end{align*}
	These results are optimal in the following sense: for any positive integer $l\geq 1$, there exists a $\mathcal{U}_{\beta^{\gamma}}$-frequently hypercyclic function for the operator $T_{\alpha}$ such that for every $0<r<1$ sufficiently large $$M_1(f,r)\lesssim\left\{\begin{array}{ll} (1-r)^{\alpha}\log_l(-\log(1-r))& \hbox{ if } \alpha<0\\1&
	\hbox{ if } \alpha>0.\end{array}\right.$$
\end{theorem}

\section{Optimal estimates: the case of the critical exponent}\label{sectioncritic}
%
%\textcolor{red}{Il reste \`a traiter les cas d'optimalit\'e pour
%\begin{enumerate}
%	\item cas $\mathcal{U}$-fhc
%	\begin{enumerate}
%\item $1<p\leq \infty$ et $\alpha=\frac{1}{\max(2,q)}$
%	\item $p=1$ et $\alpha=0$
%		\end{enumerate}
%	\item cas $\mathcal{U}_{\beta^{\gamma}}$-fhc
%	\begin{enumerate}
%		\item $1<p\leq \infty$ et $\alpha=\frac{1-\gamma}{\max(2,q)}$
%		\item $p=1$ et $\alpha=0$ : ce cas sera impliqué par le cas $\mathcal{U}$-fhc avec $p=1$ et $\alpha=0$ !
%	\end{enumerate}
%\end{enumerate}	
%}

In this section, we are going to show that the growth of $\mathcal{U}$-frequently or $U_{\beta^{\gamma}}$-frequently hypercyclic functions for $T_{\alpha}$ can be arbitrarily slow when $\alpha$ is the critical exponent. The situation will therefore be similar to the hypercyclic case for which for all $1\leq p\leq\infty$ the critical exponent is $\alpha=0$ and, according Theorem \ref{hcalpha}, the two following properties hold: for all hypercyclic function $f$ for $T_{\alpha}$, $\limsup_{r\rightarrow 1^{-}}M_p(f,r)=+\infty$ and for any function $\varphi:[0,1)\rightarrow \mathbb{R}_+$ tending to infinity as $r$ tends to $1$, there is a hypercyclic function $f$ such that $M_p(f,r)\leq \varphi(r)$. For this, we are going to adapt the constructive method used in Section \ref{subsectionconstructive}. Before we start, we establish a lemma that will be useful in the following.   

\begin{lemma}\label{keylemmaestim}
Let $(w_n)$ be an increasing sequence of positive integers such that $\frac{w_{n+1}}{w_n}\rightarrow +\infty$ as $n$ tends to infinity. Let $(a_n)$ be a bounded sequence of positive real numbers such that $\sum a_n=+\infty$. If we denote by $h:\mathbb{R}_+\rightarrow\mathbb{R}_+$ a continuous increasing function with, for all $n\in\mathbb{N}$, $h(n)=w_n$, the following estimate holds
$$\sum_{n\geq 0} a_nr^{w_n}\sim (\theta_a\circ h^{-1})\left(\frac{1}{1-r}\right)\quad\hbox{ as }r\rightarrow 1^{-},$$ 
where for all $x\in\mathbb{R_+}$, $\theta_a(x)=\sum\limits_{n\leq x} a_n$.
\end{lemma}

\begin{proof} By hypothesis the power series $S(x):=\sum\limits_{n\geq 0}a_nx^n$ has radius of convergence $1$. For all $j\geq 2$, we denote by $n_j$ the only positive integer satisfying
	$$w_{n_j+1}>j\quad\hbox{ and }\quad w_{n_j}\leq j.$$
	Then we write
	\begin{equation}\label{eq_decompo1}
	S\left(1-\frac{1}{j}\right)=\sum_{n=0}^{n_j+1}a_n\left(1-\frac{1}{j}\right)^{w_n}+\sum_{n=n_j+2}^{+\infty}a_n\left(1-\frac{1}{j}\right)^{w_n}.
	\end{equation}
On one hand, we get, thanks to the choice of $j$ and using the inequality $1-t\leq e^{-t}$,
	$$\sum_{n=n_j+2}^{+\infty}a_n\left(1-\frac{1}{j}\right)^{w_n}\leq \sum_{n=n_j+2}^{+\infty}a_n e^{-\frac{w_n}{j}}\leq M\sum_{k=2}^{+\infty}e^{-\frac{w_{n_j+k}}{j}}\leq M\sum_{k=2}^{+\infty}e^{-\frac{w_{n_j+k}}{w_{n_j+1}}},$$
where $M=\sup\vert a_n\vert$. Moreover, up to take $j$ large enough, since $\frac{w_{n+1}}{w_n}$ tends to infinity, we have, for all $k\geq 2$, 
	$\frac{w_{n_j+k}}{w_{n_j+1}}\geq 2^{k-1}$, which implies 
	 
\begin{equation}\label{eq_decompo2}\sum_{n=n_j+2}^{+\infty}a_n\left(1-\frac{1}{j}\right)^{w_n}\leq M \sum_{k=2}^{+\infty}e^{-2^{k-1}}.
\end{equation}

On the other hand, we have, for all $n=0,\dots,n_j-1$, 
$$\left(1-\frac{1}{j}\right)^{w_n}\geq \left(1-\frac{1}{j}\right)^{w_{n_j-1}}=\left(1-\frac{1}{j}\right)^{w_{n_j}\frac{w_{n_j-1}}{w_{n_j}}}\geq\left(1-\frac{1}{j}\right)^{j\frac{w_{n_j-1}}{w_{n_j}}}.$$
From this, we derive the following inequality
$$ \sum_{n=0}^{n_j+1}a_n\left(1-\frac{1}{j}\right)^{w_n}\geq \left(1-\frac{1}{j}\right)^{j\frac{w_{n_j-1}}{w_{n_j}}}\sum_{n=0}^{n_j-1}a_n.$$
Thus, we obtain
$$\left(1-\frac{1}{j}\right)^{j\frac{w_{n_j-1}}{w_{n_j}}}\frac{\theta_a(n_j-1)}{\theta_a(n_j)}\leq \frac{\sum\limits_{n=0}^{n_j+1}a_n\left(1-\frac{1}{j}\right)^{w_n}}{\theta_a(n_j)}\leq \frac{\theta_a(n_j+1)}{\theta_a(n_j)}.$$
Since $\frac{a_n}{\theta_a(n)}=1-\frac{\theta_a(n-1)}{\theta_a(n)}\rightarrow 0$ and  $\frac{w_{n_j-1}}{w_{n_j}}\rightarrow 0$, we obtain
$$\sum_{n=0}^{n_j+1}a_n\left(1-\frac{1}{j}\right)^{w_n}\sim \theta_a(n_j),\quad\hbox{ as }j\rightarrow +\infty,$$
which gives, thanks to (\ref{eq_decompo1}) and (\ref{eq_decompo2}),
\begin{equation}\label{equ_equiv_nv}
S\left(1-\frac{1}{j}\right)\sim\theta_a(n_j)\quad\hbox{ as }j\rightarrow +\infty.
\end{equation}
By construction the sequence $(n_j)$ satisfies $n_j\leq h^{-1}(j)< n_j+1$. Thus, combining this inequality with (\ref{equ_equiv_nv}) and the hypothesis $\frac{\theta_a(n+1)}{\theta_a(n)}\rightarrow 1$ again, we get
\begin{equation}\label{eq_decompo3}S\left(1-\frac{1}{j}\right)\sim\theta_a(h^{-1}(j)),\quad\hbox{ as }j\rightarrow +\infty.
\end{equation}
Let $0<r<1$ large enough with $1-\frac{1}{j}\leq r<1-\frac{1}{j+1}$, i.e $j\leq \frac{1}{1-r}<j+1$.  Clearly we have
\begin{equation}\label{eq_decompo4}
		S\left(1-\frac{1}{j}\right)\leq S(r)\leq S\left(1-\frac{1}{j+1}\right)\quad\hbox{ and }\quad h^{-1}(j)\leq h^{-1}\left(\frac{1}{1-r}\right)\leq h^{-1}(j+1).
\end{equation}
Let us recall that $n_j$ was chosen so that $n_j\leq h^{-1}(j)< n_j+1$. Assume that $h^{-1}(j)+1<h^{-1}(j+1)$. Therefore 
$$w_{n_j}\leq j<w_{n_j+1}\leq h(h^{-1}(j)+1) < h(h^{-1}(j+1))=j+1,$$
which gives a contradiction since $w_{n_j+1}$ is a positive integer. We deduce $h^{-1}(j+1)\leq h^{-1}(j)+1$ and by (\ref{eq_decompo4})
\begin{equation}\label{eq_decompo5}
h^{-1}(j)\leq h^{-1}\left(\frac{1}{1-r}\right)\leq h^{-1}(j)+1.
\end{equation}
Thanks to (\ref{eq_decompo3}), (\ref{eq_decompo4}) and (\ref{eq_decompo5}), we conclude
$$\sum_{n\geq 0} a_nr^{w_n}\sim (\theta_a\circ h^{-1})\left(\frac{1}{1-r}\right),\quad\hbox{ as }r\rightarrow 1^{-}.$$
\end{proof}

\subsection{The $\mathcal{U}$-frequently hypercyclic case}\label{subsectionfirstcase} We keep the definitions and the notations of Subsection \ref{subsectionconstructive}. Let us also consider an increasing function $h:\mathbb{R}_+\rightarrow \mathbb{R}_+$ tending to infinity such that, for any $n\in\mathbb{N}$, $h(n):=u_n\in\mathbb{N}$ and $u_{n+1}-u_n\rightarrow +\infty$ as $n$ tends to infinity.   
Let $\alpha$ be a real number. For all integer $n\geq 0,$ 
we set $I_{n}^{(u)}=\{2^{u_n},\ldots,2^{u_{n+1}}-1\}.$ Next, for $k\geq 1,$ we keep the definition of integers $\alpha_{k}$ and $\alpha^{*}_{k}$ given in Subsection \ref{subsectionconstructive}. We set $\displaystyle f_{\alpha}^{(u)}=\sum\limits_{n\geq 0}P_{n,\alpha}^{(u)}$ where the blocks $(P_{n,\alpha}^{(u)})$ are polynomials 
defined as follows, using Rudin-Shapiro polynomials given by Lemma \ref{lemma_rud_shap}, 
\begin{equation}\label{0poly_2inf}
P_{n,\alpha}^{(u)}(z)  =  \left\{\begin{array}{l} 
0 \hbox{ if } n \mbox{ is odd }\\
0 \hbox{ if } n\in \mathcal{A}_{k}\hbox{ and } 2^{u_{n-1}}<\alpha_{k}\\
z^{2^{u_n}}Q_{n}(z) \hbox{ if } n\in \mathcal{A}_{k}\hbox{ and }2^{u_{n-1}}\geq \alpha_{k}
\end{array}\right.
\end{equation}
with for $n\in \mathcal{A}_{k}$, 
$$Q_{n}^{(u)}(z)=\sum\limits_{j\in I_{n}^{(u)}}(j+1)^{-\alpha}c_{j-2^{u_n}}^{(k)}z^{j-2^{u_n}}$$ 
where the sequence $(c_{j}^{(k)})$ denotes the sequence of 
the coefficients of the polynomial $\displaystyle p_{\lfloor \frac{2^{u_n}}{\alpha_{k}} \rfloor}(z^{\alpha_k})\tilde{q_{k}}(z)$. 
We also set $\displaystyle f^{*(u)}_{\alpha}=\sum\limits_{n\geq 0}P^{*(u)}_{n,\alpha}$ where the blocks $(P^{*(u)}_{n,\alpha})$ are polynomials 
defined as follows, using polynomials given by Lemma \ref{lemma_rud_shap},  
\begin{equation}\label{0poly_12}
P^{*(u)}_{n,\alpha}(z) = \left\{\begin{array}{l} 
0 \hbox{ if } n \hbox{ is odd }\\
0 \hbox{ if } n\in \mathcal{A}_{k}\hbox{ and } 2^{u_{n-1}}<\alpha^{*}_{k}\\
\displaystyle z^{2^{u_n}}Q_{n}^{*(u)}(z)\hbox{ if } n\in \mathcal{A}_{k}\hbox{ and } 2^{u_{n-1}}\geq\alpha^{*}_{k},
\end{array}\right.
\end{equation}
with, for $n\in \mathcal{A}_{k}$,
$$Q_{n}^{*(u)}(z)=\sum\limits_{j\in I_{n}^{(u)}}(j+1)^{-\alpha}c_{j-2^{u_n}}^{(k)}z^{j-2^{u_n}}$$ 
where the sequence $(c_{j}^{(k)})$ denotes the sequence of 
the coefficients of the polynomial $\displaystyle p^{*}_{\lfloor \frac{2^{u_n}}{\alpha^{*}_{k}} \rfloor}(z^{\alpha^{*}_{k}})\tilde{q_{k}}(z)$.
\vskip2mm

\noindent For $1\leq p\leq\infty$, we denote by $\alpha_c$ the critical exponent $\alpha_c=\frac{1}{\max(2,q)}$.

\begin{lemma}\label{0Conv1} We have, for any $0<r<1$,
\begin{align*}M_p(P_{n,\alpha_c}^{(u)},r)\lesssim r^{2^{u_n}}&\hbox{ if } 2\leq p\leq\infty,\quad M_p(P^{*(u)}_{n,\alpha_c},r)\lesssim r^{2^{u_n}}\hbox{ if } 1< p<2,\\&\hbox{ and }\quad\quad M_1(P^{*(u)}_{n,0},r)\lesssim r^{2^{u_n}}l_k.
\end{align*}
\end{lemma}

\begin{proof}  For $p>1$, it suffices to argue along the same lines as the proof of Lemma \ref{Conv1} replacing the sequence $(2^{n})$ by $(2^{u_n})$.\\
	Now let us consider the case $p=1$ (hence $\alpha_c=0$). We can write, keeping in mind that$q_k=\tilde{q}_k$ for $\alpha=0$,
	$$\begin{array}{rcl}M_1(P^{*(u)}_{n,0},r)&\leq&\displaystyle\frac{r^{2^{u_n}}}{2\pi}\int_0^{2\pi}\left\vert Q_n^{*(u)}(r e^{it})\right\vert dt\\&\leq&\displaystyle
	\frac{r^{2^{u_n}}}{2\pi}\int_0^{2\pi}\vert \sum_{j\in I_n^{(u)}} c_{j-2^{u_n}}^{(k)}(r e^{it})^{j-2^{u_n}}\vert dt\\&\leq&\displaystyle
	\frac{r^{2^{u_n}}}{2\pi}\int_0^{2\pi}\left\vert   p^{*}_{\lfloor \frac{2^{u_n}}{\alpha^{*}_{k}} \rfloor}((re^{it})^{\alpha^{*}_{k}})\tilde{q_{k}}(re^{it})\right\vert dt\\&\leq&\displaystyle
	r^{2^{u_n}} \Vert p^{*}_{\lfloor \frac{2^{u_n}}{\alpha^{*}_{k}} \rfloor}\Vert_1	
	\Vert q_k\Vert_{\infty}\\&\lesssim&\displaystyle r^{2^{u_n}}l_k.
	\end{array}
	$$   
\end{proof}

From Lemma \ref{keylemmaestim} and \ref{0Conv1}, we deduce the rate of growth of the functions $f_{\alpha_c}^{(u)}$ and $f^{*(u)}_{\alpha_c}.$ We begin by the case $p\ne 1$. 

\begin{lemma}\label{01estim1falpha} Let $1<p\leq\infty$. Under the preceding definitions and assumptions, the following estimates hold: for all $0<r<1$,
\begin{align*}&\quad\quad\quad M_p(f_{\alpha_c}^{(u)},r)\lesssim  h^{-1}\left(\frac{-\log(1-r)}{\log(2)}\right)\hbox{ if }2\leq p\leq\infty\\&
\hbox{ and } \quad
M_p(f_{\alpha_c}^{*(u)},r)\lesssim h^{-1}\left(\frac{-\log(1-r)}{\log(2)}\right)\hbox{ if }1< p<2.
\end{align*}
\end{lemma}

\begin{proof} Let $2\leq p\leq\infty$. Combining Lemma \ref{0Conv1} with triangle inequality, we get $$\displaystyle M_p(f_{\alpha_c}^{(u)},r)\lesssim \sum_{n\geq 0}r^{2^{u_n}}.$$ 
By hypothesis $2^{u_{n+1}-u_n}\rightarrow +\infty$ as $n$ tends to infinity. We apply Lemma \ref{keylemmaestim} with $w_n=2^{u_n}$ and $a_n=1$ and we obtain, for $0<r<1$, 
	$$ M_p(f_{\alpha_c}^{(u)},r)\lesssim h^{-1}\left(\frac{-\log(1-r)}{\log(2)}\right).$$
For the case $1<p<2$, the proof works along the same lines.		
\end{proof}

\noindent Now we are interested in the specific case $p=1$. 

\begin{lemma}\label{specialcasep1} 
%	For any continuous and increasing function $g:\mathbb{R_+}\rightarrow \mathbb{R_+}$ tending to infinity as $t$ goes to infinity such that $g(t+1)\leq C g(t)$ for some $C>0$, 
	There is a function of the form $f_{0}^{*(u)}$ such that 
$$M_1(f_{0}^{*(u)},r)\lesssim \left(h^{-1}\left(\frac{-\log(1-r)}{\log(2)}\right)\right)^2.$$
\end{lemma}

\begin{proof} Without loss of generality we can assume that $\alpha_k^*>1+\lfloor m_k\rfloor$ where $m_k$ is the least real number such that $g(\log(\alpha_k^*)/\log(2))>l_k$. Observe that, for all $k\geq 1$, for any $n\in\mathcal{A}_k$ with $2^{u_n-1}\geq \alpha_k^*$, we have $h^{-1}(u_n)\geq h^{-1}(u_{n-1})\geq h^{-1}(\log(\alpha_k^*)/\log(2))$. Taking into account Lemma \ref{0Conv1} and the inequality $1-t\leq e^{-t}$, we get, for any $1-\frac{1}{2^{u_j}}\leq r<1-\frac{1}{2^{u_{j+1}}},$ ($j\geq 1$) 
	
$$\begin{array}{rcl}M_1(f_{0}^{*(u)},r)&\leq&\displaystyle \sum_{n\geq 1}M_1(P^{*(u)}_{n,0},r)\\&\lesssim&
\displaystyle\sum_k\sum_{n\in\mathcal{A}_k; 2^{u_n-1}\geq \alpha_k^*}\left(1-\frac{1}{2^{u_{j+1}}}\right)^{2^{u_n}}l_k
\\&\lesssim&\displaystyle
\sum_k\sum_{n\in\mathcal{A}_k; 2^{u_n-1}\geq \alpha_k^*} e^{-2^{u_n-u_{j+1}}}l_k\frac{h^{-1}(u_{n})}{h^{-1}(\log(\alpha_k^*)/\log(2))}\\&\lesssim&\displaystyle
\sum_{n=1}^{j+1} e^{-2^{u_n-u_{j+1}}} h^{-1}(u_{n})\\&\lesssim&\displaystyle
(j+1)h^{-1}(u_{j+1})=(j+1)^2.
\end{array}$$
Since $2^{u_j}\leq \frac{1}{1-r}<2^{u_{j+1}}$, we find $j\leq h^{-1}(\frac{-\log(1-r)}{\log(2)})$ and we get
$$M_1(f_{0}^{*(u)},r)\lesssim \left(h^{-1}\left(\frac{-\log(1-r)}{\log(2)}\right)\right)^2.$$
\end{proof}

Now we are going to prove that the functions $f_{\alpha_c}^{(u)}$ and $f^{*(u)}_{\alpha_c}$ are $\mathcal{U}$-frequently hypercyclic for $T_{\alpha_c}$.

\begin{proposition}\label{0propfhcdb} For $p\geq 2$ (resp. $1\leq p<2$), the function $f_{\alpha_c}^{(u)}$ (resp. $f^{*(u)}_{\alpha_c}$) is a $\mathcal{U}$-frequently hypercyclic vector for the operator $T_{\alpha_c}$.
\end{proposition}

\begin{proof}
	We only prove the frequent hypercyclicity of $f_{\alpha_c}^{(u)}$ for the operator $T_{\alpha_c}$. We do not repeat the 
	details for $f^{*(u)}_{\alpha_c}$: it will be enough to make the appropriate modifications. 
	\vskip2mm
	
	Let $k$ be a large enough integer. Let us consider $n\in\mathcal{A}_{k}$ such that $2^{u_{n-1}}\geq \alpha_{k}.$ We consider $\mathcal{B}_{n}$ the set of $s$ in 
	$I_n^{(u)}$ such that the coefficient $z^{s}$ in the polynomial $z^{2^{u_n}}p_{\lfloor \frac{2^{u_n}}{\alpha_{k}} \rfloor}(z^{\alpha_{k}})$ is equal to $1$ and we denote by $\displaystyle T_{k}=\left\{s: s\in \mathcal{B}_{n},\ n\in\mathcal{A}_{k}, \  2^{u_{n-1}}\geq \alpha_{k}\right\}.$ \\
	Observe that $\max(\mathcal{B}_{n})\leq 2^{u_n+1}$ and since at least half of the coefficients of $p_{\lfloor \frac{2^{u_n}}{\alpha_{k}} \rfloor}$ being $+1$, we get
%	\begin{equation}\label{0denshdcons}\frac{\#\{j\leq \max(\mathcal{B}_{n});j\in T_{k}\}}{\max(\mathcal{B}_{n})}
%	\geq \frac{2^{u_n-1}}{2^{u_n+1}}=\frac{1}{4},
%	\end{equation}
\begin{equation}\label{0denshdcons}\frac{\#\{j\leq \max(\mathcal{B}_{n});j\in T_{k}\}}{\max(\mathcal{B}_{n})}
\geq \frac{\frac{1}{2}(\frac{2^{u_n}}{\alpha_k}-2)-1}{2^{u_n+1}}\rightarrow\frac{1}{4\alpha_k},
\end{equation}
which implies
	$$\overline{d}(T_k)>0.$$
	Then let $\alpha$ be a real number and let $k$ be in $\mathbb{N}$. Now let us consider $s\in\mathcal{B}_{n}$ with $n\in\mathcal{A}_{k}$ satisfying 
	$2^{u_{n-1}}\geq\alpha_{k}.$ As in Proposition \ref{propfhcdb}, by construction we get  
	$$\sup_{\vert z\vert = 1-\frac{1}{l_{k}}} \vert T_{\alpha_c}^{s}(f_{\alpha_c}^{(u)})(z)-q_{k}(z)\vert \lesssim \frac{1}{l_{k}},$$ 
	provided that $k$ is chosen large enough. This allows to obtain the frequent hypercyclicity of $f_{\alpha_c}^{(u)}$.
\end{proof}

Combining Lemma \ref{01estim1falpha} with Proposition \ref{0propfhcdb} we obtain the following result.

\begin{theorem}\label{optimalp1ufhc} Let $1\leq p\leq \infty$ and $\alpha_c=\frac{1}{\max(2,q)}$. Then, for any function $\varphi:[0,1)\rightarrow\mathbb{R}_+$ with $\varphi(r)\rightarrow +\infty$ as $r\rightarrow 1^{-}$, there is a function $f$ in $H(\mathbb{D})$ with
	$$M_p(f,r)\lesssim \varphi(r),\quad\hbox{for }0<r<1\hbox{ sufficiently close to }1,$$
	that is $\mathcal{U}$-frequently hypercyclic for $T_{\alpha_c}$.
\end{theorem}

\begin{proof} We begin by the case $p>1$. Without loss of generality, we can assume that the function $\varphi$ is a continuous increasing function that can be written, for all $0<r<1$, $\varphi(r)=\psi \left(\frac{1}{1-r}\right)$ where $\psi$ is a continuous  increasing function with, for all $n\in\mathbb{N}$, $u_n:=\psi^{-1}(n)\in \mathbb{N}$ and $u_{n+1}-u_{n}\rightarrow +\infty$. Thus Lemma \ref{01estim1falpha} and Proposition \ref{0propfhcdb} ensure that, for all $1<p\leq\infty$, there is a function $f$ in $H(\mathbb{D})$ with
$$M_p(f,r)\lesssim \psi \left(\frac{-\log(1-r)}{\log(2)}\right)\lesssim \psi\left(\frac{1}{1-r}\right)=\varphi(r)$$	
that is $\mathcal{U}$-frequently hypercyclic for $T_{\alpha_c}$.\\

Now we deal with the case $p=1$. Without loss of generality, we can assume that $\varphi$ is a continuous increasing function that can be written, for all $0<r<1$, $\varphi (r)=\left(\psi \left(\frac{1}{1-r}\right)\right)^2$ where $\psi$ is a continuous and increasing function such that, for all $n\in\mathbb{N}$, $u_n:=\psi^{-1}(n)\in \mathbb{N}$ and $u_{n+1}-u_{n}\rightarrow +\infty$. Applying Lemma \ref{specialcasep1} and Proposition \ref{0propfhcdb}, we find a function $f\in H(\mathbb{D})$ with 
$$M_1(f,r)\lesssim \left(\psi \left(\frac{-\log(1-r)}{\log(2)}\right)\right)^2\lesssim \left(\psi\left(\frac{1}{1-r}\right)\right)^2=\varphi(r)$$
that is $\mathcal{U}$-frequently hypercyclic for $T_{0}$.\\	
The proof is complete.
\end{proof}

\subsection{The $\mathcal{U}_{\beta^{\gamma}}$-frequently hypercyclic case} We keep the definitions and the notations of Subsection \ref{subsectionfirstcase}. We modify the definitions of polynomials $P_{n,\alpha}^{(u)}$ and $P_{n,\alpha}^{*(u)}$ as follows:
\begin{equation}\label{2poly_2inf}
P_{n,\alpha}^{(u)}(z)  =  \left\{\begin{array}{l} 
0 \hbox{ if } n \mbox{ is odd }\\
0 \hbox{ if } n\in \mathcal{A}_{k}\hbox{ and } 2^{u_{n-1}}<\alpha_{k}\\
z^{2^{u_n}}Q_{n}(z) \hbox{ if } n\in \mathcal{A}_{k}\hbox{ and }2^{u_{n-1}}\geq \alpha_{k}
\end{array}\right.
\end{equation}
with for $n\in \mathcal{A}_{k}$, 
$$Q_{n}^{(u)}(z)=\sum\limits_{j\in I_{n}^{(u)}}(j+1)^{-\alpha}c_{j-2^{u_n}}^{(k)}z^{j-2^{u_n}}$$ 
where the sequence $(c_{j}^{(k)})$ denotes the sequence of 
the coefficients of the polynomial 
$ \displaystyle p_{\lfloor \frac{2^{u_n(1-\gamma)}}{\alpha_{k}} \rfloor}(z^{\alpha_k})\tilde{q_{k}}(z)$. 
 
\begin{equation}\label{2poly_12}
P^{*(u)}_{n,\alpha}(z) = \left\{\begin{array}{l} 
0 \hbox{ if } n \hbox{ is odd }\\
0 \hbox{ if } n\in \mathcal{A}_{k}\hbox{ and } 2^{u_{n-1}}<\alpha^{*}_{k}\\
\displaystyle z^{2^{u_n}}Q_{n}^{*(u)}(z)\hbox{ if } n\in \mathcal{A}_{k}\hbox{ and } 2^{u_{n-1}}\geq\alpha^{*}_{k},
\end{array}\right.
\end{equation}
with, for $n\in \mathcal{A}_{k}$,
$$Q_{n}^{*(u)}(z)=\sum\limits_{j\in I_{n}^{(u)}}(j+1)^{-\alpha}c_{j-2^{u_n}}^{(k)}z^{j-2^{u_n}}$$ 
where the sequence $(c_{j}^{(k)})$ denotes the sequence of 
the coefficients of the polynomial $\displaystyle p^{*}_{\lfloor \frac{2^{u_n(1-\gamma)}}{\alpha^{*}_{k}} \rfloor}(z^{\alpha^{*}_{k}})\tilde{q_{k}}(z)$.
\vskip2mm

Let $1<p\leq\infty$. Set $\alpha_c=\frac{1-\gamma}{\max(2,q)}$.

\begin{lemma}\label{2Conv1} We have, for any $0<r<1$ and all $n\in\mathbb{N}$,
	$$M_p(P_{n,\alpha_c}^{(u)},r)\lesssim r^{2^{u_n}}\hbox{ if }2\leq p\leq\infty\quad\hbox{ and }\quad
	M_p(P_{n,\alpha_c}^{*(u)},r)\lesssim r^{2^{u_n}}\hbox{ if }1< p<2.$$
\end{lemma}

\begin{proof}  It suffices to argue along the same lines as the proof of Lemma \ref{Conv1} replacing the sequence $(2^{n})$ by $(2^{u_n})$. 
\end{proof}

From Lemma \ref{keylemmaestim} and \ref{2Conv1}, we deduce the rate of growth of the functions $f_{\alpha_c}^{(u)}$ and $f^{*(u)}_{\alpha_c}.$ 

\begin{lemma}\label{21estim1falpha} Let $1<p\leq\infty$ and $\alpha_c=\frac{1-\gamma}{\max(2,q)}$. Then, for all $0<r<1$, 
$$M_p(f_{\alpha_c}^{(u)},r)\lesssim h^{-1}\left(\frac{-\log(1-r)}{\log(2)}\right)\hbox{ if }2\leq p\leq\infty,$$
and 
$$M_p(f_{\alpha_c}^{*(u)},r)\lesssim h^{-1}\left(\frac{-\log(1-r)}{\log(2)}\right)\hbox{ if }1< p<2.$$
\end{lemma}

Now we are going to prove that the functions $f_{\alpha_c}^{(u)}$ and $f^{*(u)}_{\alpha_c}$ are $U_{\beta^{\gamma}}$-frequently hypercyclic for $T_{\alpha_c}$.

\begin{proposition}\label{2propfhcdb} For $p\geq 2$ (resp. $1<p<2$), the function $f_{\alpha_c}^{(u)}$ (resp. $f^{*(u)}_{\alpha_c}$) 
	is a $U_{\beta^{\gamma}}$-frequently hypercyclic vector for the operator $T_{\alpha_c}$.
\end{proposition}

\begin{proof}
	We only prove the frequent hypercyclicity of $f_{\alpha_c}^{(u)}$ for the operator $T_{\alpha_c}$. We do not repeat the 
	details for $f^{*(u)}_{\alpha_c}$: it will be enough to make the appropriate modifications. 
	\vskip2mm
	
	Let $k$ be a large enough integer. Let us consider $n\in\mathcal{A}_{k}$ such that $2^{u_{n-1}}\geq \alpha_{k}.$ We consider $\mathcal{B}_{n}$ the set of $s$ in 
	$I_n^{(u)}$ such that the coefficient $z^{s}$ in the polynomial $z^{2^{u_n}}p_{\lfloor \frac{2^{u_n(1-\gamma)}}{\alpha_{k}} \rfloor}(z^{\alpha_{k}})$ is equal to $1$ and we denote by $\displaystyle T_{k}=\left\{s: s\in \mathcal{B}_{n},\ n\in\mathcal{A}_{k}, \  2^{u_{n-1}}\geq \alpha_{k}\right\}.$ \\
	Observe that $\max(\mathcal{B}_{n})\leq 2^{u_n}+\lfloor2^{u_n(1-\gamma)}\rfloor$ and since at least half of the coefficients of $p_{\lfloor \frac{2^{u_n(1-\gamma)}}{\alpha_{k}} \rfloor}$ being $+1$, we get
%	
%	\begin{equation}\label{2denshdcons}\frac{\sum\limits_{j\leq \max(\mathcal{B}_{n});\atop j\in T_{k}}e^{j^{\gamma}}}{\sum\limits_{j\leq \max(\mathcal{B}_{n})}e^{j^{\gamma}}}\geq 
%	\frac{\sum\limits_{j=2^{u_n}}^{2^{u_n}+2^{-1}\lfloor2^{u_n(1-\gamma)}\rfloor}e^{j^{\gamma}}}{\sum\limits_{j=1}^{2^{u_n}+\lfloor2^{u_n(1-\gamma)}\rfloor}e^{j^{\gamma}}}
%	=\frac{\sum\limits_{j=1}^{2^{u_n}+2^{-1}\lfloor2^{u_n(1-\gamma)}\rfloor}e^{j^{\gamma}}}{\sum\limits_{j=1}^{2^{u_n}+\lfloor2^{u_n(1-\gamma)}\rfloor}e^{j^{\gamma}}}-
%	\frac{\sum\limits_{j=1}^{2^{u_n}-1}e^{j^{\gamma}}}{\sum\limits_{j=1}^{2^{u_n}+\lfloor2^{u_n(1-\gamma)}\rfloor}e^{j^{\gamma}}}.
%	\end{equation}

\begin{equation}\label{2denshdcons}\frac{\sum\limits_{j\leq \max(\mathcal{B}_{n});\atop j\in T_{k}}e^{j^{\gamma}}}{\sum\limits_{j\leq \max(\mathcal{B}_{n})}e^{j^{\gamma}}}\geq 
\frac{\sum\limits_{j=2^{u_n}}^{2^{u_n}+\lfloor2^{-1}\lfloor\frac{2^{u_n(1-\gamma)}}{\alpha_k}-1\rfloor\rfloor}e^{j^{\gamma}}}{\sum\limits_{j=1}^{2^{u_n}+\lfloor2^{u_n(1-\gamma)}\rfloor}e^{j^{\gamma}}}
=\frac{\sum\limits_{j=1}^{2^{u_n}+\lfloor2^{-1}\lfloor\frac{2^{u_n(1-\gamma)}}{\alpha_k}-1\rfloor\rfloor}e^{j^{\gamma}}}{\sum\limits_{j=1}^{2^{u_n}+\lfloor2^{u_n(1-\gamma)}\rfloor}e^{j^{\gamma}}}-
\frac{\sum\limits_{j=1}^{2^{u_n}-1}e^{j^{\gamma}}}{\sum\limits_{j=1}^{2^{u_n}+\lfloor2^{u_n(1-\gamma)}\rfloor}e^{j^{\gamma}}}.
\end{equation}

	Clearly we have
\begin{align*}&\left(2^{u_n}+\lfloor2^{-1}\lfloor\alpha_k^{-1}2^{u_n(1-\gamma)}-1\rfloor\rfloor\right)^{\gamma}-\left(2^{u_n}+\lfloor2^{u_n(1-\gamma)}\rfloor\right)^{\gamma}\rightarrow -\gamma(1-\frac{1}{2\alpha_k}),\\
	&\left(2^{u_n}-1\right)^{\gamma}-\left(2^{u_n}+\lfloor2^{u_n(1-\gamma)}\rfloor\right)^{\gamma}\rightarrow -\gamma,
\end{align*}
which implies, using similar estimations as those of the proof of Lemma \ref{diffdensbeta}, 
$$\frac{\sum\limits_{j=1}^{2^{u_n}+\lfloor 2^{-1}\lfloor\frac{2^{u_n(1-\gamma)}}{\alpha_k}-1\rfloor\rfloor}e^{j^{\gamma}}}{\sum\limits_{j=1}^{2^{u_n}+\lfloor2^{u_n(1-\gamma)}\rfloor}e^{j^{\gamma}}}-
\frac{\sum\limits_{j=1}^{2^{u_n}-1}e^{j^{\gamma}}}{\sum\limits_{j=1}^{2^{u_n}+\lfloor2^{u_n(1-\gamma)}\rfloor}e^{j^{\gamma}}}\rightarrow e^{-\gamma}(e^{1/2\alpha_k}-1), \hbox{ as }n\rightarrow +\infty.$$
%
%	Clearly we have
%	\begin{align*}&\left(2^{u_n}+\frac{\lfloor2^{u_n(1-\gamma)}\rfloor}{2}\right)^{\gamma}-\left(2^{u_n}+\lfloor2^{u_n(1-\gamma)}\rfloor\right)^{\gamma}\rightarrow -\frac{\gamma}{2},\\
%	&\left(2^{u_n}-1\right)^{\gamma}-\left(2^{u_n}+\lfloor2^{u_n(1-\gamma)}\rfloor\right)^{\gamma}\rightarrow -\gamma,
%	\end{align*}
%	which implies, using similar estimations as those of the proof of Lemma \ref{diffdensbeta}, 
%	$$\frac{\sum\limits_{j=1}^{2^{u_n}+2^{-1}\lfloor2^{u_n(1-\gamma)}\rfloor}e^{j^{\gamma}}}{\sum\limits_{j=1}^{2^{u_n}+\lfloor2^{u_n(1-\gamma)}\rfloor}e^{j^{\gamma}}}-
%	\frac{\sum\limits_{j=1}^{2^{u_n}-1}e^{j^{\gamma}}}{\sum\limits_{j=1}^{2^{u_n}+\lfloor2^{u_n(1-\gamma)}\rfloor}e^{j^{\gamma}}}\rightarrow e^{-\gamma/2}(1-e^{-\gamma/2}), \hbox{ as }n\rightarrow +\infty.$$
	Hence the inequality (\ref{2denshdcons}) ensures that 
	$$\overline{d}_{\beta^{\gamma}}(T_k)>0.$$
	Then let $\alpha$ be a real number and let $k$ be in $\mathbb{N}$. Now for  $s\in\mathcal{B}_{n}$ with $n\in\mathcal{A}_{k}$ satisfying 
	$2^{u_{n-1}}\geq\alpha_{k}$, as in Proposition \ref{propfhcdb}, by construction we get  
	$$\sup_{\vert z\vert = 1-\frac{1}{l_{k}}} \vert T_{\alpha_c}^{s}(f_{\alpha_c}^{(u)})(z)-q_{k}(z)\vert \lesssim \frac{1}{l_{k}},$$ 
	provided that $k$ is chosen large enough. This allows to obtain the frequent hypercyclicity of $f_{\alpha_c}^{(u)}$.

\end{proof}

Combining Lemma \ref{21estim1falpha} with Proposition \ref{2propfhcdb} we obtain the following result.

\begin{theorem}\label{optimalubetagamma} Let $0<\gamma<1$. Let $1\leq p\leq \infty$ and $\alpha_c=\frac{1-\gamma}{\max(2,q)}$. Then, for any function $\varphi:[0,1)\rightarrow\mathbb{R}_+$, with $\varphi(r)\rightarrow +\infty$ as $r\rightarrow 1^{-}$, there is a function $f$ in $H(\mathbb{D})$ with
	$$M_p(f,r)\lesssim \varphi(r),\quad\hbox{ for }0<r<1\hbox{ sufficiently close to }1,$$
	that is $\mathcal{U}_{\beta^{\gamma}}$-frequently hypercyclic for $T_{\alpha_c}$.
\end{theorem}

\begin{proof} For $p=1$, we have $\alpha_c=0$ and the result is given by Theorem \ref{optimalp1ufhc}. Now let $p>1$. Without loss of generality, we can assume that the $\varphi$ is a continuous increasing function such that, for all $0<r<1$, $\varphi(r)=\psi \left(\frac{1}{1-r}\right)$ where $\psi$ is continuous and increasing with, for all $n\in\mathbb{N}$, $u_n:=\psi^{-1}(n)\in \mathbb{N}$ and $u_{n+1}-u_{n}\rightarrow +\infty$. Thus Lemma \ref{21estim1falpha} and Proposition \ref{2propfhcdb} ensure that, for all $1<p\leq\infty$, there is a function $f$ in $H(\mathbb{D})$ with
	$$M_p(f,r)\lesssim \psi \left(\frac{-\log(1-r)}{\log(2)}\right)\lesssim \varphi (r)$$
	that is $\mathcal{U}_{\beta^{\gamma}}$-frequently hypercyclic for $T_{\alpha_c}$. The proof is complete.
\end{proof}
\vskip5mm

From Theorems \ref{hcalpha}, \ref{optimalp1ufhc} and \ref{optimalubetagamma}, we can state the following result that unifies what happens in the critical case, given by the critical exponent, for the $L^p$ growth of $\mathcal{U}_{\beta^{\gamma}}$-frequently hypercyclic functions for $T_{\alpha}$, when $\gamma$ belongs to $[0,1]$. 

\begin{theorem}\label{mainexpcritical} Let $0\leq \gamma \leq 1$ and $1\leq p\leq \infty$.  Then, for any function $\varphi:[0,1)\rightarrow\mathbb{R}_+$, with $\varphi(r)\rightarrow +\infty$ as $r\rightarrow 1^{-}$, there is a function $f$ in $H(\mathbb{D})$ with $M_p(f,r)\lesssim \varphi(r)$ that is $\mathcal{U}_{\beta^{\gamma}}$-frequently hypercyclic for $T_{\alpha_c}$ where $\alpha_c=\frac{1-\gamma}{\max(2,q)}$.
\end{theorem}

\vskip10mm

\noindent \textbf{Concluding remark.} In summary, thanks to Theorems \ref{thmubetaopti}, \ref{udfhcalpha_no_criticp1}, \ref{optimalp1ufhc} and \ref{optimalubetagamma}, the results given by Theorems \ref{ufhcalpha_no_critic}, \ref{ufhcalpha_no_criticp1} and \ref{ubetafhcalpha_no_critic} are optimal.

\end{document}